\documentclass[12pt]{amsart}

\usepackage[all]{xy}
\usepackage{amssymb}
\usepackage{amsmath}
\usepackage{stmaryrd}
\usepackage{bbm}
\usepackage{txfonts}
\usepackage{tikz-cd}
\usepackage{url}
\usepackage{etoolbox}

\usepackage{mathtools}

\makeatletter
\patchcmd{\@addmarginpar}{\ifodd\c@page}{\ifodd\c@page\@tempcnta\m@ne}{}{}
\makeatother
\reversemarginpar

\usepackage[OT2,T1]{fontenc}
\DeclareSymbolFont{cyrletters}{OT2}{wncyr}{m}{n}
\DeclareMathSymbol{\Sha}{\mathalpha}{cyrletters}{"58}

\newcommand{\segment}[2]{\subsection{#2}\label{#1}}
\newcommand{\ssegment}[2]{\subsubsection{#2}\label{#1}}

\theoremstyle{definition}
\newtheorem*{prop*}{Proposition}

\theoremstyle{definition}
\newtheorem*{thm*}{Theorem}

\theoremstyle{definition}
\newtheorem*{cor*}{Corollary}

\theoremstyle{definition} 
\newtheorem{ssprop}[subsubsection]{Proposition}
\newcommand{\ssProposition}[2]{\begin{ssprop} \label{#1} 
{#2} \end{ssprop}}

\theoremstyle{definition} 
\newtheorem{sprop}[subsection]{Proposition}
\newcommand{\sProposition}[2]{\begin{sprop} \label{#1} 
{#2} \end{sprop}}

\theoremstyle{definition} 
\newtheorem{sconj}[subsection]{Conjecture}

\theoremstyle{definition} 
\newtheorem{sthm}[subsection]{Theorem}

\theoremstyle{definition} 
\newtheorem{ssthm}[subsubsection]{Theorem}

\theoremstyle{definition} 
\newtheorem{sslm}[subsubsection]{Lemma}

\theoremstyle{definition} 
\newtheorem{slm}[subsection]{Lemma}
\newcommand{\sLemma}[2]{\begin{slm} \label{#1} 
{#2} \end{slm}}

\theoremstyle{definition} 
\newtheorem{scor}[subsection]{Corollary}

\theoremstyle{definition} 

\newcommand{\ssCorollary}[2]{\begin{scor} \label{#1} 
{#2} \end{sscor}}

\theoremstyle{definition} 
\newtheorem{ssconj}[subsubsection]{Conjecture}

\theoremstyle{definition}
\newtheorem*{conj*}{Conjecture}

\theoremstyle{definition} 
\newtheorem{sscond}[subsubsection]{Condition}

\theoremstyle{definition} 
\newtheorem{ssrmk}[subsubsection]{Remark}
\newcommand{\ssRemark}[2]{\begin{ssrmk} \label{#1} 
{#2} \end{ssrmk}}

\let\oldmarginpar\marginpar
\renewcommand\marginpar[1]{\-\oldmarginpar[\raggedleft\footnotesize #1]%
{\raggedright\footnotesize #1}}

\makeatletter
\renewcommand*\env@matrix[1][\arraystretch]{%
  \edef\arraystretch{#1}%
  \hskip -\arraycolsep
  \let\@ifnextchar\new@ifnextchar
  \array{*\c@MaxMatrixCols c}}
\makeatother

\newcommand{\Rep}{\operatorname{\bf{Rep}}}
\newcommand{\Vect}{\operatorname{\bf{Vect}}}

\newcommand{\from}{\leftarrow}
\newcommand{\xto}{\xrightarrow}
\newcommand{\xfrom}{\xleftarrow}
\newcommand{\surj}{\twoheadrightarrow}
\newcommand{\xyto}[2]{\underset{#2} {\overset{#1}\rightrightarrows}}

\renewcommand{\Im}{\operatorname{Im}}

\newcommand{\Spec}{\operatorname{Spec}}

\newcommand{\Cone}{\operatorname{Cone}}
\newcommand{\tot}{\operatorname{tot}}
\newcommand{\Hom}{\operatorname{Hom}}

\newcommand{\Isom}{\operatorname{Isom}}

\newcommand{\Aut}{\mathrm{Aut}\,}

\newcommand{\m}[1]{\mathrm{#1}}
\newcommand{\fk}[1]{\mathfrak{#1}}

\newcommand{\bb}[1]{\mathbb{#1}}

\newcommand{\ka}{\kappa}

\newcommand{\Si}{\Sigma}
\newcommand{\ze}{\zeta}
\newcommand{\ga}{\gamma}
\newcommand{\al}{\alpha}
\newcommand{\be}{\beta}
\newcommand{\om}{\omega}
\newcommand{\Om}{\Omega}

\newcommand{\Qp}{{\QQ_p}}
\newcommand{\Zp}{{\ZZ_p}}

\newcommand{\Qbar}{{\overline{\QQ}}}

\newcommand{\ZZ}{\bb{Z}}
\newcommand{\CC}{\bb{C}}

\newcommand{\LL}{\bb{L}}
\newcommand{\MM}{\bb{M}}
\newcommand{\DD}{\bb{D}}
\newcommand{\NN}{\bb{N}}
\newcommand{\QQ}{\bb{Q}}
\newcommand{\RR}{\bb{R}}
\newcommand{\Rr}{\mathcal{R}}
\newcommand{\PP}{\bb{P}}
\newcommand{\pP}{\fk{p}}

\newcommand{\Gg}{\mathcal{G}}
\newcommand{\Hh}{\mathcal{H}}
\renewcommand{\AA}{\bb{A}}
\newcommand{\Yy}{\mathcal{Y}}

\newcommand{\Oo}{\mathcal{O}}

\newcommand{\Tt}{\mathcal{T}}
\newcommand{\Aa}{\mathcal{A}}

\newcommand{\Mm}{\mathcal{M}}
\newcommand{\Xx}{\mathcal{X}}

\newcommand{\Zz}{\mathcal{Z}}

\newcommand{\Cpx}{\operatorname{Cpx}}
\newcommand{\Bdry}{\operatorname{Bdry}}

\newcommand{\inv}{^{-1}}

\newcommand{\areq}{\ar@{=}}
\newcommand{\suphook}{\ar@{^(->}}
\newcommand{\subhook}{\ar@{_(->}}

\newcommand{\smses}[6]
{
\[
\xymatrix{
1 \ar[r] &
#1 \ar[r]_-{#2} &
#3 \ar[r]_-{#4} &
#5 \ar[r] \ar@/_1.5pc/[l]_-{#6} &
1
}
\]
}

\newcommand{\To}{\;\longrightarrow\;}

\newcommand{\dR}{{\rm {dR}}}

\newcommand{\et}{{\textrm {\'et}}}
\newcommand{\mot}{{\rm {mot}}}

\newcommand{\un}{\m{un}}

\newcommand{\Spt}{\operatorname{Spt}}
\newcommand{\Cpl}{\operatorname{Cpx}}
\newcommand{\PSh}{\operatorname{Psh}}
\newcommand{\DA}{\operatorname{DA}}

\newcommand{\Sus}{\operatorname{Sus}}
\newcommand{\Ho}{\operatorname{Ho}}

\newcommand{\Cotors}{\operatorname{Cotors}}

\newcommand{\eff}{\m{eff}}

\newcommand{\CMM}{\m{Mdga} }

\newcommand{\Mdga}{\m{Mdga}}
\newcommand{\hrat}{\operatorname{ C } }
\newcommand{\Sm}{\operatorname{Sm}}
\newcommand{\op}{^\m{op}}

\newcommand{\ccM}{\operatorname{cc\MM}}
\newcommand{\DMdga}{\operatorname{DMdga}}

\newcommand{\one}{\mathbbm{1}}
\newcommand{\onef}{\one^f}

\newcommand{\SptSSA}{\Spt_S^\fk{S}(\Aa)}

\newcommand{\Sing}{\m{Sing}}
\newcommand{\cMon}{\operatorname{cMon}}

\newcommand{\DAcell}{\operatorname{DA_{cell}}}
\newcommand{\cell}{\m{cell}}

\newcommand{\tototo}{\underset{\to}{\overset{\to}{\to}}}

\newcommand{\cHfd}{\operatorname{cHfd}}

\newcommand{\Cogpd}{\operatorname{Cogpd}}

\newcommand{\dga}{\operatorname{cdga}}
\newcommand{\Ddga}{\operatorname{Ddga}}
\newcommand{\Ind}{\operatorname{Ind}}
\newcommand{\hocolim}{\operatorname{hocolim}}

\newcommand{\DDMdga}{\operatorname{\DD Mdga}}
\newcommand{\DDdga}{\operatorname{\DD dga}}

\newcommand{\C}{\operatorname{C}}

\newcommand{\colim}{\operatorname{colim}}

\newcommand{\can}{\m{can}}

\newcommand{\Sh}{\operatorname{Sh}}

\newcommand{\Gal}{\operatorname{Gal}}

\newcommand{\Set}{\operatorname{Set}}

\newcommand{\Cell}{\operatorname{Cell}}

\newcommand{\Uni}{\operatorname{Uni}}
\newcommand{\Comod}{\operatorname{Comod}}

\newcommand{\RRe}{\operatorname{\RR e}}

\newcommand{\DMod}{\operatorname{DMod}}
\newcommand{\DcoComod}{\operatorname{D^{co}Comod}}

\newcommand{\holim}{\operatorname{holim}}

\newcommand{\DTM}{\operatorname{DTM}}

\newcommand{\TM}{\operatorname{TM}}

\title[Rational motivic path spaces]{Rational motivic path spaces and Kim's relative unipotent section conjecture}

\author[Dan-Cohen, Schlank]{Ishai Dan-Cohen and Tomer Schlank}

\thanks{I.D.C. supported by Priority Program 1786 of the Deutsche Forschungsgemeinschaft {\it Homotopy theory and algebraic geometry} and by ISF grant 87590031.} 
\thanks{T.S. supported by an Alon Fellowship.}

\date{\today}

\begin{document}

\maketitle

\raggedbottom
\SelectTips{cm}{11}

\begin{abstract}
We develop the foundations of commutative algebra objects in the category of motives, which we call ``motivic dga's''. Work of White and of Cisinski-D\'eglise provides us with a suitable model structure. This enables us to reconstruct the unipotent fundamental group of a pointed scheme from the associated augmented motivic dga, and provides us with a factorization of Kim's relative unipotent section conjecture into several smaller conjectures with a homotopical flavor. 

\bigskip

\end{abstract}

\setcounter{tocdepth}{1}
\tableofcontents

\setcounter{section}{-1}

%%%%%%%%%%%%%%%%%
\section{Introduction}%%%%%%%%%%%
%%%%%%%%%%%%%%

\segment{}{}
Let $Z$ be an open integer scheme (open in $\Spec \Oo_K$, $K$ a number field). For $p \in Z$ a closed point, we let $Z_p$ denote the complete local scheme at $p$. Let $X$ be a smooth curve over $Z$ with \'etale divisor at infinity. The classical method of \emph{Chabauty} for studying the set $X(Z)$ of $Z$-points (at least when $X$ is proper) constructs a subspace $X(Z_p)_\m{Chab}$ of $X(Z_p)$ which contains $X(Z)$. This subspace of \emph{Chabauty points} is defined by analytic equations. Due to contributions by a range of authors, especially Coleman, it has turned out to be quite computable via a certain theory of $p$-adic integration, producing effective bounds on the number of $Z$-points in a certain range of cases. Outside of this range however, $X(Z_p)_\m{Chab} = X(Z_p)$, so Chabauty's method breaks down completely. Moreover, even within this range, $X(Z_p)_\m{Chab}$ is rarely equal to $X(Z)$. 

Better in this respect is Grothendieck's section conjecture which is expected to exactly classify the set of $Z$-points via the torsors of \'etale paths that connect them to a fixed base-point. Kim's conjecture, announced in \cite{nabsd} may be regarded as a variant of the section conjecture which is \emph{unipotent} on the one hand, and \emph{relative} on the other. \emph{Unipotent} refers to the fact that the profinite fundamental group of Grothendieck's conjecture is replaced by its prounipotent completion. This involves a significant loss of information, and, correspondingly, to an overabundance of torsors. Kim then restricts attention to those torsors which are \emph{locally geometric}, that is, come from $p$-adic points for some auxiliary prime $p$ of $Z$ (assumed to be of good reduction). The resulting conjecture is about extending local points to global points, hence \emph{relative}.

Kim's conjecture is also closely related to Chabauty's method. His original papers on the subject \cite{kimi, kimii} were devoted to the construction of a nested sequence of analytic subspaces
\[
\cdots \subset X(Z_p)_3 \subset X(Z_p)_2
\subset X(Z_p)_1 = X(Z_p)_\m{Chab},
\]
defined by $p$-adic \emph{iterated} integrals, and his foundational theorem there may be summarized by the inclusion
\[
\tag{*}
X(Z) \subset X(Z_p)_\m{Kim} : = \bigcap_n X(Z_p)_n.
\]
A series of works on the subject by Kim himself, as well as by a number of other authors \cite{BalakAppendix, CKtwo, CKIII, KimTangential, KimMassey, BalakDog, BalBesMul, BalBesMulComp, BalBesColGr} have shown that the analytic spaces $X(Z_p)_n$ (which are finite when not equal to all of $X(Z_p)$), are amenable to explicit computation in a range of cases beyond the Chabauty bound. The conjecture states that the inclusion (*) above is in fact an equality. Part of our goal here is to explain our belief that this conjecture should be more tractable than the section conjecture.

\segment{}{}
If Kim's approach to integral points is a classification in terms of torsors under the unipotent fundamental group, our work provides an a priori finer classification in terms of torsors under a certain rational loop-space. So Kim's homotopy classes of paths derive, in our approach, from \emph{actual} paths, in a rational, motivic sense. 

Our main theorem may be summarized roughly as follows. 

\begin{ssthm}[Factorization of Kim's cutter, preliminary version]
\label{thmfact}
We assume for technical reasons that $X$ is affine (but we expect this assumption to be spurious). We then have the following inclusions:
\begin{align*}
X(Z) &\subset X(Z_p)_{\mbox{motivically global}} \\ 
	& \subset X(Z_p)_{\mbox{pathwise motivically global}} \\
	&  \subset X(Z_p)_
{ \mbox{pathwise motivically global up to p-adic \'etale homotopy} } \\
	&  \subset X(Z_p)_{\mbox{Kim}}. \\
\end{align*}
Moreover, if $X$ is \emph{mixed Tate}, then we may replace ``$p$-adic \'etale homotopy'' with ``motivic homotopy''.
\end{ssthm}

\noindent
See theorem \ref{factbig} for the precise statement. 

%%%%%%%%%%%%%%%%%%%%
\segment{}{Review of Chen's theorem}%%%
%%%%%%%%%%%%%%%%%%%%%%

In order to explain what the phrases in subscript mean (including the subscript ``Kim''), we first recall Chen's theorem, following Wojtkowiak's account \cite{Wojtkowiak}, embellished somewhat by a more homotopical language. We start by recalling a homotopy-theoretic approach to the topological fundamental group. Let $\Xx:= X(\CC)$ and let
\[
pt \xto{x} \Xx \xfrom{y} pt
\]
be two points of $\Xx$. Then there's a formula from topology:
\[
x \times_\Xx^{ho} y = {_yP_x},
\]
the space of paths $x \to y$, and
\[
_y\pi_x = \pi_0({_yP_x}),
\]
the space of homotopy classes of paths. In this description, the topological fundamental group has a natural de Rham analog. Let $\Om^\bullet(X_\QQ)$ denote the \textit{de Rham dga} associated to $X_\QQ$ (this is the derived global sections of the sheaf of dga's $\Om^\bullet_{X_\QQ/\QQ}$ in a suitable sense, c.f. Olsson \cite[8.1]{OlssonBar}), and assume $x$, $y$ come from rational points
\[
\Spec \QQ \xto{x} X \xfrom{y} \Spec \QQ.
\]
Then $x$, $y$ induce augmentations
\[
\QQ \xfrom{x^*} \Om^\bullet(X_\QQ) \xto{y^*} \QQ.
\]
Let
\[
_yB_x^\dR := \QQ \coprod^{ho}_{\Om^\bullet(X_\QQ)} \QQ,
\]
the homotopy pushout being taken in the model category of (commutative) dga's\footnote{
Throughout this article, all dga's are assumed to be commutative. 
}. Then
$
_yA_x^\dR := H^0({_yB_x^\dR})
$
has the structure of a commutative $\QQ$-algebra, and there are \emph{path cocomposition maps}
\[
{_zA_y^\dR}\otimes{_yA_x^\dR} \from {_zA_x^\dR}.
\]
These give
$
\pi_1^\dR(X,x): = \Spec {_xA_x^\dR}
$
the structure of a proalgebraic $\QQ$-group (which turns out to be prounipotent, c.f. theorem \ref{21d}), and they give
$
{_y\pi_x^\dR} := \Spec {_yA_x^\dR}
$
the structure of a torsor. Chen's theorem states that the $\QQ$-group $\pi_1^\dR(X,x) \otimes \CC$ is the complex prounipotent completion of $\pi_1(\Xx,x)$.

\ssegment{}{}%%%%%%%
For concreteness, we outline the construction of an explicit map
\[
_y\pi_x \to {_y\pi_x^\dR}(\CC)
\]
in terms of iterated integrals (and recall some googlable key-words along the way). This material will not be used in the sequel. The \emph{Cosimplicial model of the path space} is given by
\[
_yP_x^\bullet(X)= 
\xymatrix{
pt \ar@<+.5ex>[r]^x \ar@<-.5ex>[r]_y &
X \ar@<+.7ex>[r] \ar[r] \ar@<-.7ex>[r] &
X^2 
\ar@<+.9ex>[r] \ar@<+.3ex>[r] \ar@<-.3ex>[r] \ar@<-.9ex>[r]
&
\cdots
}
\]
with e.g.
\[
t \in X
\hspace{3ex}
\begin{matrix}
\mapsto (x,t) \\
\mapsto (t,t) \\
\mapsto (t,y).
\end{matrix}
\]
The \emph{reduced bar construction} is an explicit simplicial dga $_yB_x^{\dR, \bullet}$ with 
\[
_yB_x^{\dR, n} = \Om^\bullet_{/\QQ}(X_\QQ)^{\otimes n}
\]
which represents the derived push-out
$
_yB_x^\dR = \tot {_yB_x^{\dR, \bullet}}.
$
The K\"unneth formula gives rise to an equivalence of simplicial dga's
\[
\Om^\bullet_{/\QQ}({_yP_x^\bullet}) \xto{\sim} {_yB_x^{\dR, \bullet}}.
\]
This gives us an isomorphism
\[
\Spec H^0 \big( \tot \Om^\bullet_{/\QQ}({_yP_x^\bullet}) \big)
\xto{\sim}
{_y\pi_x^\dR}.
\]

We wish to construct a map
\[
{_y\pi_x} 
\xto{\Psi} 
\Big(
\Spec H^0 \big(\tot \Om^\bullet_{/\QQ}({_yP_x^\bullet}) \big)
 \Big)
(\CC)
\]
Regard the image of a piecewise differentiable path
\[
\ga:x \to y
\]
as a homomorphism
\[
\Psi\ga: 
H^0 \big(\tot \Om^\bullet_{/\QQ}({_yP_x^\bullet}) \big)
\to 
\CC.
\]
By the Grothendieck de Rham theorem, an element
\[
[\om] \in H^0 \tot \Om^\bullet{_yP^\bullet_x} \otimes \CC
\]
may be represented by a family
$
\om = (\om_i)_{i\in \NN}
$
with
$
\om_i \in \Om^i_{\m{smooth}} (X^i_\CC).
$
On the other hand, 
$
\ga: x \to y
$
gives rise to a piecewise differentiable simplex 
$
\ga^i \in C_i^{\Sing}(\Xx^i)
$
given by
\[
\xymatrix 
@ R=2ex @C=7ex
{
\RR^i
\\
\{0 \le t_1 \le \dots \le t_i \le 1 \}
\ar@{}[u] | \bigcup
\ar[r]^-{(\ga, \dots, \ga)}
&
X^i.
}
\]
We set
\[
\Psi\ga([\om]) = \sum_i \int_{\ga^i} \om_i.
\]

%%%%%%%%%%%%%%%%%%%%%%%%
\segment{}{Motivic reconstruction of the unipotent fundamental group}%%%%%%%%%%%%
%%%%%%%%%%%%%%%%%%%%%%%%%%
The point of departure for this project is the following observation. It is said that ``Motives cannot distinguish between points and cycles''. Kim's theory shows that the unipotent fundamental group, however, remembers enough about integral points to provide a classification relative to local points, the classification being complete whenever his conjecture is true (and there's a growing list of cases, c.f. \cite{nabsd}). The unipotent fundamental group, in turn, may be \emph{reconstructed} from the data
\[
\QQ \xfrom{y^*}
\Om^\bullet(X) \xto{x^*} \QQ,
\]
\emph{if we include the cup product}. This data should be in some sense motivic. So we retort that motives \emph{equipped with cup product} can distinguish points from cycles.

We fix a model $M(Z,\QQ)$ for the tensor triangulated category $\DA(Z,\QQ)$ of motivic complexes over $Z$ with $\QQ$ coefficients. We let $\Mdga(Z,\QQ)$ denote the category of commutative monoid objects in $M(Z, \QQ)$ and call its objects ``motivic dga's''. The familiar cohomological motives functor of Levine \cite{LevineMM}, Voevodsky \cite{VoevTrCat}, Ayoub \cite{AyoubSixI, AyoubSixII} and others, may be upgraded in an obvious way to a functor
\[
C = C^*_\mot(\;\cdot\;, \QQ): \m{Sm}_{/Z}^\m{op} \to \Mdga(Z,\QQ).
\]
Relying on work of White \cite{White} and Cisinski-D\'eglise \cite{CDHomAlg}, we show in Propositions \ref{1611290924} and \ref{lprop} that the expected classes of weak equivalences, fibrations and cofibrations endow $\Mdga(Z,\QQ)$ with the structure of a left proper cofibrantly generated model category such that the free algebra functor and the forgetful functor form a Quillen adjunction. 

In the setting of motivic dga's, integral points $x,y \in X(Z)$ give rise to augmentations
\[
\one \xfrom{y^*}
 C(X) \xto{x^*} \one.
\]
We may then form the homotopy pushout
\[
{_y B_x} = \one \coprod ^{ho}_{\hrat(X)} \one
\]
in $\Mdga(Z,\QQ)$. The resulting objects ${_y B_x}(X)$ are endowed with  \emph{path cocomposition} morphisms. We refer to ${_yB_x}$ as the \emph{coordinate ring of rational motivic paths from $x$ to $y$}, or just \emph{path ring} for short.\footnote{This construction provides a motivic source for an object whose  realizations have appeared in various places in the literature. For instance, a filtered-$\phi$ realization is constructed in Kim--Hain \cite{KimHain}. A fairly different construction which gives rise to an object which should play a similar role appears in Pridham \cite[\S4.5]{pridhamenhanced}.}

When $X$ is mixed Tate, we may take cohomology with respect to the motivic t-structure of Levine \cite{LevineVanishing} and compare the result with the unipotent fundamental group of Deligne-Goncharov \cite{DelGon}
\[
\pi^\un(X,x) = \Spec H^0_t \big( {_x B_x}(X) \big),
\]
a prounipotent group object of the category of mixed Tate motives which realizes to the prounipotent $\QQ$-completion of the Betti fundamental group $\pi_1\big(X(\CC), x \big)$ (depending on the choice of an embedding $K \subset \CC$), to the prounipotent $\Qp$-completion of the profinite \'etale fundamental group $\hat \pi_1(X_{\overline K}, x)$ for every prime $p$ of $Z$, and to the fundamental group of the category of unipotent connections on $X_{K}$.

%%%%%%%%%%%%%%%
\segment{factsimple}{Factorization of Kim's cutter}%%%%%%%%
%%%%%%%%%%%%%%%%%%

This slightly different way of constructing the unipotent fundamental group leads to a natural factorization of ``Kim's cutter'' \cite{CKtwo}. We define a \emph{pseudo-cotorsor} for $B_x$ to be a commutative monoid  object $P$ of the triangulated tensor category of motives $\DA(Z, \QQ)$, together with a  morphism of monoid objects
\[
B_x \otimes^\LL P \from P
\]
which is coassociative, such that
$
B_x \otimes^\LL P \xfrom{\sim} P \otimes^\LL P.
$
A pseudo-cotorsor is a \emph{cotorsor} if it satisfies a certain nonemptyness condition; as we explain below, the right condition for us is the nonvanishing
$
H^0_\et(P_{\overline K}, \Qp) \neq 0.
$
In the mixed Tate setting, we restrict attention to cotorsors which are, in an appropriate sense \emph{cellular}. 

\ssegment{diasimple}{}
We let $\DMdga = \DMdga(Z,\QQ)$ denote the homotopy category of $\Mdga(Z,\QQ)$. To fix ideas let us restrict attention to the mixed Tate case, and for simplicity, let us assume that $Z \subset \Spec \ZZ$. We also fix a base point $x \in X(Z)$. Let $\pi_1^{\TM}(Z)$ denote the fundamental group of the category of mixed Tate motives. Let $\pi_1^{[n]}(X)$ denote the quotient of $\pi_1^\un(X)$ by the $n$th step in the descending central series.

In our construction of the motivic path ring ${_yB_x}$, we may replace the \emph{arithmetic endpoint} $y \in X(Z)$ with any \emph{motivic endpoint}
\[
\xi \in \Hom_{\DMdga}(CX, \one).
\]
This means that the natural map from points to cotorsors factors through the set $\Hom_{\DMdga}(CX, \one)$ of ``motivic points'' and leads to a diagram in which the square formed by Kim's cutter becomes the union of four smaller squares: 
\[
\xymatrix
@C= -8 ex
{
%1
X(Z) \ar[rr] \ar[d]_{f_1} &&
		X(Z_p) \ar[d]^{f_1^p} \\
%2
\Hom_{\DMdga}(\hrat(X), \one) 
		\ar[rr]^-{l_1} \ar[d]_{f_2} &&
		\Hom_{\DMdga}(\hrat(X_p), \one) \ar[d]^{f_2^p} \\
%3
\Cotors_\cell (B_x(X)) \ar[d]_{f_3} \ar[rr]^-{l_2} &&
		\Cotors (B_x(X_p)) \ar[dr]^-{\al} \\
%4
H^1(\pi_1^{\TM}(Z), \pi_1^\un(X)) \ar[dr] \ar[rrr]^-{l_3} &&&
	\pi_1^\dR(X)_\Qp \ar[d] \\ 
%5
&
	H^1(\pi_1^{\TM}(Z)_\Qp, \pi_1^\un(X)^{[n]}_\Qp)
	 \ar[rr]_-{\m{Re}_{\Qp}} &&
			\pi_1^\dR(X)^{[n]}_\Qp
}
\]
In terms of this diagram, we define
\[
X(Z_p)_{\mbox{motivically global}}: = (f_1^p)\inv(\Im l_1)
\]
as the preimage in $X(Z_p)$ of the image of $l_1$, we define
\[
X(Z_p)_{\mbox{pathwise motivically global}}
\]
as the preimage in $X(Z_p)$ of $l_2$, and we define
\[
X(Z_p)_{\mbox{pathwise motivically global up to motivic homotopy}}
\]
as the preimage in $X(Z_p)$ of the image of $l_3$. The sets $X(Z_p)_n$ which define $X(Z_p)_\m{Kim}$ are different: $\m{Re}_\Qp$ is a map of $\Qp$-varieties, and we take \emph{scheme-theoretic image}, remembering nilpotent structure (in our opinion) when pulling back to $X(Z_p)$.

\ssegment{}{}%%%%%%
If we wish to define the map $f_2$ in the diagram of segment \ref{diasimple} by associating to an augmentation
\[
\xi: C(X) \to \one
\]	
the homotopy pushout 
\[
f_2(\xi) = {_\xi B_x}
= \hocolim \left(
\one \xfrom{\xi}
C(X) \xto{x^*} \one
\right)
\]
in $\Mdga$, we need to know that $H^0_\et \big( ({_\xi B_x})_{\overline K}, \Qp \big) \neq 0$. In more intuitive terms, we need to know that every two motivic points are connected by a $p$-adic \'etale path. 

Continuing down the left column, if we wish to define the map $f_3$ by associating to a cotorsor $P$ the torsor
\[
f_3(P):= \Spec H^0_t(P)
\]
of \emph{motivic homotopy classes of paths}, we need to know that $H_t^i(P) = 0 $ for $i<0$, in other words, that $P$ is concentrated in negative (homological) degrees. We also need to know that $H^0_t(P) \neq 0$ (although it will not follow that every two motivic points are connected by a motivic path, since there may be no morphisms $\QQ(0) \to \Spec H^0_t(P)$). Since $X$ is affine (this assumption now being essential), we have, moreover, $H^i_t(P) = 0$ for $i>0$, so $P$ is concentrated in degree $0$. These four facts follow from the \emph{connectedness} and \emph{concentration} theorems. 

\begin{sthm}[Connectedness]
\label{ctd} 
Let $Z$ be a Dedekind scheme whose function field $K$ has characteristic zero (an open integer scheme in our applications). Let $X$ be a smooth scheme over $Z$ such that
\[
H^0_\et(X_{\overline K}, \Qp) = \Qp
\]
(equivalently, $X$ is generically geometrically connected). Let
\[
\om, \eta \in \Hom_{\DMdga}(CX, \one)
\]
be augmentations. We assume $X$ is an affine curve. Then
\[
H^0_\et \big( {_\eta B_\om}(X)_{\overline K} , \Qp \big) \neq 0.
\]
If moreover $Z$ obeys Beilinson-Soul\'e vanishing (okay for $Z$ an open integer scheme) and $X$ is mixed Tate, then the same holds for the cohomologies $H^i_t$ associated to the mixed Tate t-structure. 
\end{sthm}

\begin{sthm}[Concentration]
\label{ctr}
Let $Z$ be a Dedekind scheme whose function field $K$ has characteristic zero. Let $X$ be a smooth affine generically geometrically connected curve over $Z$, and let $\om, \eta \in \Hom_{\DMdga}(CX, \one)$ be augmentations. Then
\[
H^i_\et \big( {_{\eta}B_\om}(X)_{\overline K}, \Qp \big) = 0
\]
for $i \neq 0$. If moreover $Z$ obeys Beilinson-Soul\'e vanishing and $X$ is mixed Tate, then the same holds for the cohomologies $H^i_t$ associated to the mixed Tate t-structure. 
\end{sthm}

A priori, the map $f_2$ sends augmentations to pseudo-cotorsors; the connectedness theorem ensures that these obey our nonemptyness condition. As we show in segment \ref{1611291401}, the concentration theorem ensures that $f_3$ sends pseudo-cotorsors to pseudo-torsors. Our nonemptyness condition on pseudo-cotorsors then implies that the result is in fact a torsor and not just a pseudo-torsor.

As a direct corollary of the concentration theorem, we find,

\begin{scor}[Pathwise motivically global theorem]
\label{pathglob}
Suppose $X$ is a smooth mixed Tate affine curve with \'etale divisor at infinity over an open integer scheme $Z$, and let $p \in Z$ be a closed point. Then the inclusion
\[
\renewcommand\arraystretch{0.9}
X(Z_p)_
{
\begin{matrix*}[l]
\mbox{pathwise} \\
\mbox{motivically} \\
\mbox{global}
\end{matrix*}
}
\subset 
X(Z_p)_
{
\begin{matrix*}[l]
\mbox{pathwise} \\
\mbox{motivically global up to} \\
\mbox{motivic homotopy}
\end{matrix*}
}
\]
is a bijection. 
\end{scor}

\segment{unpalb}{}
Returning to the diagram of segment \ref{diasimple}, the pattern of symmetry between the global column on the left and the local column on the right is broken after the third row. There are several reasons for this. For one, we currently lack a t-structure on the triangulated category of mixed Tate motives over the local scheme $Z_p$. This leads us to replace nonabelian motivic cohomology with the nonabelian syntomic cohomology
\[
H^1\big( 
\pi_1^{\m{MTF\phi}}(Z_p),
 \pi_1^\un(X)^{F\phi}
\big).
\]
Here $\pi_1^{\m{MTF\phi}}(Z_p)$ denotes the fundamental group of the category of mixed Tate filtered $\phi$ modules \cite{ChatUnv} and $\pi_1^\un(X)^\m{F\phi}$ is the filtered $\phi$ realization of the unipotent fundamental group. An appropriate realization functor defined on the triangulated level gives us a map
\[
\tag{*}
\Cotors(B_{x_{p}}) \to 
H^1\big( 
\pi_1^{\m{MTF\phi}}(Z_p),
 \pi_1^\un(X)^{F\phi}
\big).
\] 
The nonabelian syntomic cohomology is then isomorphic to $\pi_1^\dR(X)/F^0$, the Tannakian fundamental group of the category of unipotent vector bundles with integrable connection on $X_{K_p}$ modulo the 0th step in its Hodge filtration. By composing this isomorphism with the map (*) we obtain a map 
\[
\al:X(Z_p) \to \pi_1^\dR(X_{K_p})/F^0
\]
studied for instance by Kim \cite{kimii} and known as the \emph{unipotent $p$-adic Albanese map}. (This map specializes to the one labeled $\al$ in the diagram, since in the mixed Tate case we have $F^0 = 0$.) It's an important feature of Kim's construction (and another reason for breaking the symmetry) that this map is given in coordinates by $p$-adic iterated integrals \`a la Besser--Coleman \cite{Besser, Coleman} which can often be computed algorithmically.

\segment{}{}%%%%
This relationship with Coleman integrals depends on a certain comparison result. For $x,y\in X(Z_p)$,  the Frobenius action on the realization of the motivic path ring $_yB_x$ should be equal to the one induced by the Frobenius action on the the category of unipotent connections via the equivalence of categories of Chiarelotto--Le Stum \cite{ChiarLeStum} between the latter category and the category of unipotent isocrystals on the special fiber. As far as we can tell, the philosophically correct proof of this comparison result would proceed by applying theorem \ref{21d} below to certain topoi associated to an appropriate rational homotopy theory of $p$-adic mixed Hodge modules. While the affine stacks of To\"en \cite{ToenAffine} and the methods of Olsson \cite{OlssonTowards} and Pridham \cite{PridhamGalois} may provide a shortcut, we have elected to cheat by going to the $p$-adic \'etale realization (refined somewhat by the recent contribution of Scholze-Bhatt \cite{ScholzeProet}). The commutativity of the diagram above then follows by concatenating the triangle of unipotent $p$-adic Hodge theory \cite{OlssonTowards, kimii} with the diagram of segment \ref{diabig} below.

Our approach to the $p$-adic \'etale realization should be compared with that of Deligne--Goncharov \cite{DelGon}. While the latter is based on an explicit realization of the homotopy pushout via the bar construction and on an elaborate computation in constructible cohomology due to Beilinson, we turn to Morita theory and Koszul duality to understand the unipotent fundamental groupoid of an arbitrary topos. Theorem \ref{21d} (mentioned above) states the result of this investigation. 

\segment{}{}%%%%%%
This article is foundational in nature. In much of what follows, we assemble theorems of others using tried and true techniques in order to set the stage for a new avenue of investigation into Kim's conjecture. Needless to say, one of its more ambitious goals is to go in the direction of Kim's conjecture; we make no claims as to our level of optimism regarding the prospects of an actual proof. Short of a proof however, we do expect to find, in each of the sub-inclusions of theorem \ref{thmfact}, philosophical evidence (as opposed to numerical evidence) for the conjecture. Indeed, each of the four inclusions presents itself together with an obvious collection of loose threads at which to tug. As an example, the conjectured implication ``motivically global $\Rightarrow$ global'' leads to the following expectation. Let $X$, $Y$ be schemes over $Z $, let $p$ be a closed point of $Z$, and let's assume for simplicity that $Z\subset \Spec \ZZ$. Then we may consider the square
\[
\xymatrix{
\Hom(Y,X) \ar[r] \ar[d] &
 \Hom(Y_{Z_p}, X_{Z_p}) \ar[d]^f
 \\
\Hom_{\DMdga}(CX,CY) \ar[r]_-l &
\Hom_{\DMdga}(CX_{Z_p}, CY_{Z_p}). 
}
\]
We should expect that for \emph{many} schemes $X$, $Y$, we have 
\[
\tag{*}
\Hom(Y,X) = f\inv (\Im l).
\]
For instance, since it's reasonable to expect Kim's conjecture to extend (via a suitably relative formulation) to morphisms of anabelian schemes over $Z$, we should expect the equality (*) to hold at least when $X$, $Y$ are anabelian. In particular, we should expect an equality whenever $X$, $Y$ are finite \'etale over $Z$, or suitable models of hyperbolic curves. Moreover, when $X$, $Y$ are proper, we may replace $Z$ by $\Spec \QQ$. So if after this grandiose generalization we again specialize to the case $Z = \Spec \QQ$, $X = Y = \Spec k$, $k$ a finite extension of $\QQ$, $p$ a prime at which $k$ is totally split, we're led to the expectation that 
\[
\Aut_{\DMdga}(CX)= \Gal(k/\QQ).
\]
Our proof of this simple statement in motivic rational homotopy theory in \cite{nmh} may be regarded as bit of evidence for a piece of Kim's conjecture. 

\segment{}{}
In this preliminary work, we restrict attention to the \emph{mere set} of strict cotorsors in the homotopy category. In doing so, we ignore a geometric / homotopical structure which is naturally available. Indeed, there's a moduli stack of weak cotorsors together with coherent higher homotopies on the model category or infinity category level, which stands above and behind our set of cotorsors, much in the same way as the local Selmer variety stands behind the nonabelian $p$-adic \'etale cohomology set in Kim's theory. Another part of our program will be devoted to a study of functions on the stack of cotorsors and their pullbacks to $X(Z_p)$. One possible outcome would give rise to an interesting class of functions, generalizing the iterated integrals of Besser--Coleman. The opposite outcome, however --- that the resulting functions are no different than the classical iterated integrals, may enable us to use the stack of cotorsors to \emph{do} Chabauty-Kim theory on the triangulated, motivic level. This could include applications to integral points on varieties with abelian fundamental group, or to a ``Kim-theoretic'' proof of Faltings' theorem via an appropriate lift of Poitou-Tate duality (c.f. Schlank--Stojanoska \cite{SchlankGalDu, SchlankPT}), circumventing the $p$-adic regulator isomorphism conjecture of Bloch--Kato (c.f. Kim \cite{kimii}).

\subsection*{Acknowlegements}
First and foremost we would like to thank Marc Levine for his interest and his encouragement, and for many helpful conversations devoted to the topic of this paper. We would like to thank Federico Binda, Shane Kelly, Adeel Khan, Lorenzo Mantovani, Brad Drew, Fr\'ed\'eric D\'eglise, Minhyong Kim, Joseph Ayoub, and Tyler Lawson for helpful conversations and correspondence.  We would like to thank Martin Gallauer for pointing out a mistake in an earlier draft. The first author would like to thank Jochen Heinloth for his hospitality as his host at the University of Duisburg-Essen, and for his advice regarding exposition. Finally, we thank the referee for helpful comments. 

%%%%%%%%%%%%%%%
\section{Motivic dga's}%%%%%%
%%%%%%%%%%%%%%

The results of this section developed in conversations with Marc Levine and Gabriela Guzman, whose work on this topic ultimately diverged from ours.

\segment{1611290922}{}%%%%%%

%We consider $Z \subset \Spec \ZZ$ open, and $X = \PP_Z^1 \setminus D$, $D \subset \PP^1(Z)$ a rational divisor which is \'etale over $Z$. Fix $x \in X(Z)$. Fix $p \in Z$.
 Let $Z$ be a Dedekind scheme, let $\Sm_Z$ denote the category of smooth schemes over $Z$, let $\Cpx \PSh(\Sm_Z, \QQ)$ denote the category of complexes\footnote{Here and below, complexes are not assumed to be bounded.} of presheaves of $\QQ$-vector spaces. If $X$ is a smooth scheme over $Z$, we commit the usual abuse by denoting the associated presheaf again by $X$. Continuing with this abuse, we let $X\otimes \QQ$ denote the presheaf whose value on a smooth $Z$-scheme $Y$ is the $\QQ$-vector space of formal linear combinations of elements of $\Hom_Z(Y,X)$.

 We let
 \[
 \MM^\eff(Z,\QQ)
 =
 \big(
 \Cpx \PSh (\Sm_Z,\QQ),
 (\AA^1, \et)\mbox{-local}
 \big)
 \]
 denote the category of complexes of presheaves, endowed with the ($\AA^1$, \'et)\emph{-local model structure}. We recall briefly what this means, referring to \cite[\S3]{AyoubEtale} for details. We start with the projective model structure, in which the weak equivalences are defined to be the quasi-isomorphisms and the fibrations are defined to be the degree-wise surjections. We left-localize with respect to the class of morphisms which induce isomorphisms of cohomology presheaves after sheafification with respect to the \'etale topology. We then left localize with respect to the class of morphisms 
\[
\AA^1_X\otimes \QQ[n] 
\to
X \otimes \QQ[n] 
\]
for $X \in \Sm_Z$ and $n\in \ZZ$.
 
We set 
\[
T:= 
\frac{\PP^1_Z \otimes \QQ[0] }
 {\infty \otimes \QQ[0]}.
\]

\segment{bubble}{}%%%%%%%
We let $\MM(Z, \QQ)$ denote the category of symmetric $T$-spectra in $\MM^\eff(Z,\QQ)$ as in Hovey \cite[\S7]{HoveySpectra}, or, equivalently but more simply, the category of commutative $T$-spectra of Ayoub \cite[\S4]{AyoubComodules}, endowed with the \emph{stable} $(\AA^1,\et)$-local model structure, as in Ayoub \cite[\S3]{AyoubEtale}. There's a natural left Quillen functor 
\[
\Sus^0_T: \MM^\eff(Z,\QQ) \to \MM(Z,\QQ),
\]
and we set
\[
\QQ(0) := \Sus^0_T(\QQ[0]).
\]
This makes $\MM(Z,\QQ)$ into a monoidal model category with unit object $\QQ(0)$ and tensor product given by the natural (levelwise) tensor product. The \emph{triangulated category of \'etale motives over $Z$} is given by
\[
\DA^\et(Z,\QQ) := \Ho \MM(Z,\QQ).
\]
We will generally drop the decoration $`\et'$.

The \emph{homological motives functor}
\[
C_*^\m{DA} =
C_*^\m{DA}(-,\QQ)
: \Sm_Z \to \MM(Z,\QQ) \to \DA(Z,\QQ)
\]
is given on an object $X\in \Sm_Z$ by
\[
C_*^\m{DA}(X,\QQ) = 
\Sus^0_T(X\otimes \QQ[0]).
\]
The \emph{Tate object} is given by $\QQ(1):= \Sus^0_T(T[-2])$.

\segment{cd}{}%%%%%%%
The Bousfield localizations that intervene in the above construction are made explicit via descent-theoretic techniques in Cisinski-D\'eglise \cite{CDHomAlg}. We review their construction.

\ssegment{cd1}{}%%%%%%
We let $\Aa$ denote an abelian category endowed with a closed symmetric monoidal structure. We let $\Cpx \Aa$ denote the category of (always unbounded) complexes in $\Aa$. We let $K(\Aa)$ denote the \textit{homotopy category} in the sense of standard homological algebra: the objects are those of $\Cpx \Aa$ and the morphisms are equivalence classes with respect to chain-homotopy. We let $D(\Aa)$ denote the derived category: the objects are again those of $\Cpx \Aa$, but the category has been localized with respect to quasi-isomorphisms. 

 If $E$ is an object and $n$ is an integer, we define
\[
S^nE:= E[n]
\]
and 
\[
D^{n+1}E = ( \cdots \to 0 \to E \xto{Id} E \to 0 \to \cdots)
\]
in cohomological degrees $n-1$, $n$, so that there's a natural map
\[
\Bdry(E,n):S^nE \to D^{n+1}E.
\]

\ssegment{cd2}{}%%%%%%
Fix an essentially small collection $\Gg$ of objects of $\Aa$ and an essentially small collection $\Hh$ of complexes in $\Aa$. A morphism in the category $\Cpx\Aa$ of complexes in $\Aa$ is a \emph{$\Gg$-cofibration} if it is contained in the smallest class of maps in $\Cpx\Aa$ closed under pushouts, transfinite compositions and retracts, and containing the inclusions $\Bdry(E,n)$ for $n\in \ZZ$ and $E\in \Gg$. A complex $C$ is \emph{$\Gg$-cofibrant} if
\[
0\to C
\] 
is a $\Gg$-cofibration; \emph{$\Gg$-local} if for any $E$ in $\Gg$ and $n\in \ZZ$, the map
\[
\Hom_{K(\Aa)}(E[n], C) \to \Hom_{D(\Aa)}(E[n], C)
\]
is bijective; and \emph{$\Hh$-flasque} if for all $n\in\ZZ$ and $H\in\Hh$, 
\[
\Hom_{K(\Aa)}(H, C[n])=0.
\]
The pair $(\Gg, \Hh)$ forms a \emph{descent structure on $\Aa$} if
\begin{description}
\item[DS1] every element of $\Hh$ is $\Gg$-cofibrant and acyclic, and
\item[DS2] every $\Hh$-flasque complex is $\Gg$-local.
\end{description}
A descent structure is said to \emph{flat} if 
\begin{description}
\item[FlDS1] for any $X \in \Gg$, and any quasi-isomorphism of complexes $f$, the morphism of complexes $id_X \otimes f$ is again a quasi-isomorphism, and
\item[FlDS2] $\Gg$ contains the unit object $\one$ and is essentially closed under tensor product.
\end{description}

\ssegment{cd3}{}%%%%%%%
Let $\Tt$ be a set of complexes in $\Aa$. We say that $\Tt$ is a \emph{flat localizing family} if
\begin{description}
\item[FlLF1] the elements of $\Tt$ are $\Gg$-cofibrant, and
\item[FlLF2] for all $E \in \Gg$ and $T \in \Tt$, $E \otimes T \in \Tt$.
\end{description}

\ssegment{cd4}{}%%%%%
In addition to the data ($\Aa$, $\Gg$, $\Hh$, $\Tt$) above, we fix a $\Gg$-cofibrant complex $S$. We then have the category $\Spt_S^\fk{S}(\Aa)$ of symmetric $S$-spectra in $\Cpx(\Aa)$. We construct a model structure on $\Spt_S^\fk{S}(\Aa)$ in four steps, following the usual rubric. 

\begin{enumerate}
\item
We start with the \emph{$\Gg$-projective model structure} on $\Cpx(\Aa)$: the equivalences are defined to be the quasi-isomorphisms, and the cofibrations are defined to be the $\Gg$-cofibrations \cite[Theorem 2.5]{CDHomAlg}. 

\item
We define the \emph{$\Tt$-local $\Gg$-projective} model structure on $\Cpx \Aa$ in terms of the $\Gg$-projective structure as follows. We define a morphism of complexes
\[
X \to Y
\]
to be a $\Tt$-equivalence if for any $\Hh \cup \Tt$-flasque complex $K$, the map
\[
\Hom_{D(\Aa)}(Y,K) \to \Hom_{D(\Aa)}(X,K)
\]
is bijective. For the $\Tt$-local $\Gg$-projective structure we take $\Tt$-equivalences for weak equivalences, and $\Gg$-cofibrations for cofibrations. This model structure is equal to the Bousfield localization of the $\Gg$-projective model structure by the class of morphisms
\[
0 \to T[n]
\]
for $T \in \Tt$ and $n\in \ZZ$ \cite[Prop 4.3]{CDHomAlg}. We will call the associated fibrations \emph{$\Tt$-local $\Gg$-projective fibrations}. 

\item
We define the \emph{$\Tt$-local $\Gg$-projective structure} on $\Spt_S^\fk{S}(\Aa)$ by taking levelwise $\Tt$-equivalences for weak equivalences, and levelwise $\Tt$-local $\Gg$-projective fibrations for fibrations. We then have the associated homotopy category
\[
D_\Tt(\Aa,S) := \Spt_S^\fk{S}(\Aa)[\{\Tt\mbox{-equivalences}\}\inv].
\]

\item
Combining Proposition 4.3 with segment 7.8 of \cite{CDHomAlg}, we define an $S$-spectrum $E$ to be an \emph{$\Om^\infty$-spectrum} if it is fibrant for the $\Tt$-local $\Gg$-projective structure on $\Spt_S^\fk{S}(\Aa)$ and if, moreover, for each $n \in \NN$, the bonding map
\[
E_n \to \underline \Hom_{\Cpx\Aa}(S, E_{n+1})
\]
is a $\Tt$-equivalence. We define a morphism of symmetric $S$-spectra 
\[
A \to B
\]
to be a \emph{stable $\Tt$-equivalence} if for every $\Om^\infty$-spectrum $E$, the induced map
\[
\Hom_{D_\Tt(\Aa,S)}(A,E) \from \Hom_{D_\Tt(\Aa,S)}(B,E)
\]
is bijective. A morphism of symmetric $S$-spectra is a \emph{stable $\Gg$-cofibration} if it is a cofibration for the $\Tt$-local $\Gg$-projective model structure. We define the \emph{stable $\Tt$-local $\Gg$-projective model structure} on $\Spt_S^\fk{S}(\Aa)$ by taking the stable $\Tt$-equivalences for weak equivalences and the stable $\Gg$-cofibrations for cofibrations. 
\end{enumerate}

\ssegment{dc5}{}%%%%%%%
Recall the \emph{monoid axiom} on a monoidal model category $M$ of Schwede-Shipley \cite{SchwedeShipley}: every 
\[
\big((\mbox{trivial cofibration}) \otimes M \big) \mbox{-cell object}
\]
is a weak equivalence.

\begin{ssthm}[7.13 and 7.24 of Cisinski-D\'eglise \cite{CDHomAlg}]
\label{dc6}%%%% 
Let $\Aa$ be an abelian category endowed with a flat descent structure $(\Gg, \Hh)$, a flat localizing family $\Tt$, and a $\Gg$-cofibrant complex $S$. Then the stable $\Tt$-local $\Gg$-projective structure makes $\Spt_S^\fk{S}(\Aa)$ into a proper, cellular, symmetric monoidal model category which obeys the monoid axiom. Moreover, an object of $\SptSSA$ is fibrant if and only if it is an $\Om^\infty$-spectrum.
%\footnote{We note that the proof of the monoid axiom is a d\'evisage which executes itself automatically by regarding the tensor product as a left Quillen functor from the projective to the injective model structure, a technique which appears in various places in the literature.}
\end{ssthm}

\ssegment{dc7}{}%%%%%%%
We apply the above constructions with $\Aa$ the category of \'etale sheaves of $\QQ$-vector spaces on $\Sm_Z$, $\Gg$ the collection of sheaves of the form $X \otimes \QQ$ ($X\in \Sm_Z$), $\Hh$ the family of complexes of sheaves of the form 
\[
\Cone(\Xx \otimes \QQ \to X \otimes \QQ)
\]
where
$
\Xx \to X
$
is an \'etale hypercovering and $\Xx \otimes \QQ$ is the complex of sheaves of $\QQ$-vector spaces obtained by taking alternating sums of boundary maps, $\Tt$ the family of complexes
\[
\tag{A}
\cdots \to 
0 \to 
\AA^1_X \otimes\QQ \to 
X\otimes \QQ \to 
0 \to 
\cdots
\]
in degrees $0$, $1$,
and $S$ the complex
\[
\tag{S}
\cdots \to 0 \to \QQ(0) \to 
(\AA^1 \setminus \{0\}) \otimes \QQ \to 
0 \to \cdots
\]
in degrees $-1$, $0$. We then have a canonical Quillen equivalence 
\[
\MM(Z, \QQ) = \Spt_S^\fk{S}(\Aa)
\]
with the stable $\Tt$-local $\Gg$-fibrant model structure on the latter. This is well understood by experts, so we limit ourselves to a brief indication of the steps of the verification. Let $\m{proj}$ denote the projective model structure on
$
\Cpx \PSh (\Sm_Z,\QQ)
$.
 Let $\Gg'$ denote the collection of presheaves of the form
$
X\otimes \QQ
$
 for $X \in \Sm_Z$. Then the class of cofibrations in 
\[
\big(\Cpx \PSh (\Sm_Z,\QQ), \m{proj}\big)
\]
 is the smallest class of morphisms closed under pushouts, transfinite compositions and retracts and containing the boundary inclusions $\Bdry(E,n)$ for $n\in \ZZ$ and $E \in \Gg$. In other words
\[
\big(
\Cpx \PSh (\Sm_Z,\QQ),
 \m{proj}
 \big)
=
\big(
\Cpx \PSh (\Sm_Z,\QQ),
 \Gg'\mbox{-proj}
 \big).
\]
Let $a$ denote the sheafification functor and let $b$ denote its right adjoint, which is the inclusion of sheaves in presheaves. One checks that these form a Quillen adjunction 
\[
a: 
\big(
\Cpx \PSh (\Sm_Z,\QQ),
 \Gg'\mbox{-proj}
 \big)
\leftrightarrows
\big(
\Cpx \Sh (\Sm_{Z,\et},\QQ), \Gg\mbox{-proj}
\big)
:b
\]
(where the left adjoint is on the left). Let $W_\et$ denote the class of \'etale-local quasi-isomorphisms of complexes of presheaves and let $L_{W_\et}$ denote Bousfield localization with respect to $W_\et$. Using the universal mapping property of the latter, one checks that the Quillen adjunction above gives rise to a Quillen equivalence
\[
\tag{1}
\big(
\Cpx \PSh (\Sm_Z,\QQ), 
L_{W_\et}(\Gg'\mbox{-proj})
\big)
\leftrightarrows
\big(
\Cpx \Sh (\Sm_{Z, \et},\QQ),
 \Gg\mbox{-proj}
 \big).
\]

To compare the $\AA^1$-localization used by Ayoub \cite[\S3]{AyoubEtale} in the construction of $\MM^\eff(Z,\QQ)$ (recalled in segment \ref{1611290922} above) with the one used by Cisinski-D\'eglise \cite{CDHomAlg} (segment \ref{cd4} above applied with $\Tt$ equal to the set of complexes \ref{dc7}(A)) in their construction of the $\Tt$-local $\Gg$-projective model structure, we note simply that in a stable model category, localization with respect to a morphism $f:A \to B$ is equivalent to localization with respect to $0\to C$ where $C$ denotes the cone of $f$. Combining with (1), we obtain a Quillen equivalence
\[
\tag{2}
\MM^\m{eff}(Z,\QQ) 
\leftrightarrows
\big(
\Cpx \Sh (\Sm_{Z, \et},\QQ),
\Tt\mbox{-loc }
 \Gg\mbox{-proj}
 \big).
\]

To compare the two stabilizations, we compute in
\[
\DA^\m{eff}(Z,\QQ) = \Ho \MM^\m{eff}(Z,\QQ).
\]
There, the standard Zariski covering of $\PP^1$ gives rise to a distinguished triangle
\[
(\AA^1 \setminus \{0\}) \otimes \QQ[0]
\to
\QQ(0) \oplus \QQ(0) 
\to 
\PP^1 \otimes \QQ[0]
\to
(\AA^1 \setminus \{0\}) \otimes \QQ[1].
\]
Hence, applying the octahedral axiom to the sequence of morphisms
\[
\QQ(0) \oplus \QQ(0) 
\to
\QQ(0)
\to
\PP^1\otimes \QQ[0]
\]
we obtain the equivalence
\[
\frac{\PP^1 \otimes \QQ[0]}{\QQ(0)}
=
\frac{(\AA^1 \setminus \{0\}) \otimes \QQ[0]}{\QQ(0)} [1].
\]
The last complex is quasi-isomorphic (as a complex of presheaves) to the complex (S). Hence the two stabilizations produce Quillen equivalent model categories. 

\begin{scor}
\label{monax}
The motivic model category $\MM(Z, \QQ)$ is proper, cellular, symmetric monoidal, and obeys the monoid axiom as well as the commutative monoid axiom of White \cite{White}. 
\end{scor}

\begin{proof}
Special case of theorem \ref{dc6}, except for the commutative monoid axiom. The latter follows from the (not commutative) monoid axiom since $\MM(Z,\QQ)$ is $\QQ$-linear, so that the quotient of a morphism $f$ by a $\Sigma_n$-action is a retract of $f$.
\end{proof}

\segment{1611290924}{Definition}%%%%%%%%
We define the \emph{rational motivic model category} $\Mdga = \Mdga(Z, \QQ)$ to be the category of commutative monoids in $\MM(Z,\QQ)$. We define a morphism in $\CMM$ to be a \emph{weak equivalence} if the associated morphism in $\MM(Z,\QQ)$ is a weak equivalence, and a \emph{cofibration} if the associated morphism in $\MM(Z,\QQ)$ is a cofibration.

\begin{sprop}
\label{p1}%%%%%%
With these definitions, $\Mdga(Z,\QQ)$ is a cofibrantly generated model category.
\end{sprop}

\begin{proof}
This is an application of Theorem 3.2 of \cite{White} in view of corollary \ref{monax}. 
\end{proof}

\begin{slm}%%%%
\label{cof-fl}
Cofibrant objects of $\MM(Z,\QQ)$ are flat, that is, if $M$ is a cofibrant object of $\MM(Z,\QQ)$ and $f:A\to B$ is a weak equivalence, then $id_M \otimes f$ is again a weak equivalence.
\end{slm}

\begin{proof}
Although not stated explicitly by Cinsinski-D\'eglise, this is essentially a feature of their construction. In the notation of segment \ref{cd}, we are to show that for any $\Om^\infty$-spectrum $E$, the induced map
\[
\Hom_{D_\Tt(\Aa,S)} \Big( M \underset{\SptSSA}\otimes A, \; E \Big)
 \from \Hom_{D_\Tt(\Aa,S)} \Big( M  \underset{\SptSSA}\otimes B, \; E \Big)
\]
is bijective. (We've decorated the tensor products to emphasize the fact that they're \emph{not} derived.) By adjunction with inner Hom, it's equivalent to show that the map
\[
\Hom_{D_\Tt(\Aa,S)} \big( A, \underline \Hom_{\SptSSA}(M,E) \big)
 \from \Hom_{D_\Tt(\Aa,S)} \big( B, \underline \Hom_{\SptSSA}(M,E) \big)
\]
is bijective.  So it's enough to show that the inner Hom
\[
\underline \Hom_{\SptSSA}(M,E)
\]
is an $\Om^\infty$-spectrum. Since $M$ is cofibrant for the stable $\Tt$-local $\Gg$-projective structure on $\SptSSA$ and $E$ is fibrant for the same structure, it follows from the monoidal model category structure on $\MM(Z,\QQ)$ that the inner Hom is fibrant for the same structure, hence an $\Om^\infty$-spectrum as hoped. \end{proof}

\segment{scm0}{}%%%%
Recall that that the \emph{pushout product} $f \oblong g$ of morphisms
$f: X \to X'$, $g:Y \to Y'$ is by definition the induced morphism
\[
X \otimes Y' \underset{X \otimes Y} {\small \coprod} X' \otimes Y
 \to X' \otimes Y'
\]

\begin{slm}
\label{scm1}%%%%%
The motivic model category $\MM(Z,\QQ)$ obeys the \emph{strong commutative monoid axiom} of White \cite[Definition 3.4]{White}: if $h$ is a cofibration then
\[
h^{\medcirc n}:= h^{\oblong n}/\Si_n
\]
is again a cofibration for all $n \ge 1$; moreover, if $h$ is a trivial cofibration, then so is $h^{\medcirc n}$.
\end{slm}

\begin{proof}
$\MM(Z,\QQ)$ being monoidal \textit{means} (in part) that the pushout-product of cofibrations is a cofibration, and the pushout product of trivial cofibrations is a trivial cofibration. It follows by induction that if $h$ is a cofibration (respectively a trivial cofibration), so is $h^{\oblong n}$. Since $\MM(Z,\QQ)$ is $\QQ$-linear, $h^{\medcirc n}$ is a retract of $h^{\oblong n}$, hence again a cofibration (respectively a trivial cofibration). 
\end{proof}

\segment{hmon0}{}%%%%%
Given morphisms
\[
\xymatrix{
X \ar[d]_f \ar[r]^g & A \\
Y
}
\]
let ${^g f}$ denote the pushout of $f$ along $g$. If $X/\Mm$ is the category of morphisms from an object $X$ in a model category $\Mm$, then a morphism in $X/\Mm$ is defined to be a \emph{weak equivalence} if the associated morphism in $\Mm$ is a weak equivalence. Recall that a monoidal model category $\Mm$ is \emph{h-monoidal} if for any cofibration
\[
f: X \to Y,
\]
and any object $Z$,
\begin{description}
\item[hM1] the co-base-change functor
\[
(f\otimes id_Z)_!: (X \otimes Z) / \Mm \to (Y\otimes Z)/\Mm
\]
preserves weak equivalences
(in other words, for any morphism
\[
g: X \otimes Z \to A,
\]
pushout along $^g (f\otimes id_Z)$ preserves weak equivalences), and
\item[hM2] if, moreover, $f$ is a weak equivalence, then so is $f \otimes id_Z$.
\end{description}
(C.f. definition 4.15 of White \cite{White}.)

\begin{slm}%%%%%
\label{hmon}
The motivic model category $\MM(Z,\QQ)$ is h-monoidal. 
\end{slm}

\begin{proof}
Combine Proposition 1.13 of Batanin-Berger \cite{BataninBerger} with Propositions 7.13, 7.19, and 7.23 of Cinsinski-D\'eglise \cite{CDHomAlg}.
\end{proof}

\begin{sprop}%%%%%
\label{lprop}
The category $\Mdga(Z,\QQ)$ of motivic dga's is left proper. 
\end{sprop}

\begin{proof}
Domains of generating cofibrations are cofibrant. Indeed, before stabilizing, every object is cofibrant, and it's straightforward to check that this property is preserved by stabilization. By lemmas \ref{cof-fl}, \ref{scm1}, and \ref{hmon}, and corollary \ref{monax}, theorem 4.17 of White \cite{White} applies.\footnote{Our category is compactly generated as a mere category and weak equivalences are preserved by filtered colimits. Therefor $\MM(Z,\QQ)$ is a compactly generated model category.}
\end{proof}

%%%%%%%%%%%%%%%%%%%%%%%%%%%%%%
%\section{The motivic dga associated to a smooth scheme}%
%%%%%%%%%%%%%%%%%%%%%%%%%%%%%%

\segment{d1}{}%%%
Using the model structure of proposition \ref{p1}, we may fix once and for all a fibrant replacement
\[
\one \to \onef
\]
of the unit object in $\Mdga$. There's a natural functor 
\[
\operatorname{C_*} = 
\operatorname{C_*^\mot }( \;\cdot\; , \QQ): \Sm_Z \to \ccM
\]
to the category of cocommutative comonoids in $\MM(Z,\QQ)$ given by
\[
C_*X := \Sus^0_T(X \otimes \QQ[0])
\]
with counit
\[
C_*X \to C_*Z = \QQ(0)
\]
induced by the structure map of $X$, and comultiplication
\[
C_*X \to C_*X \otimes C_*X
\]
induced by the diagonal of $X$. In terms of the object $\onef$ and the functor $C_*$, we have the following theorem.

\begin{sthm}%%%%
\label{hrat}
The formula 
\[
\hrat X =
 \operatorname{C^*_\mot}(X, \QQ):=
  \underline \Hom_{\MM(Z,\QQ)}(C_*(X, \QQ),\onef)
\] 
defines a functor 
\[
\C = \C_*^\mot ( \;\cdot\; , \QQ) : \Sm\op_Z \to \Mdga(Z,\QQ).
\]
\end{sthm}

\begin{proof}
For $X \in \Sm_Z$, the inner hom above has the structure of a (strictly!) commutative monoid in $\MM(Z,\QQ)$. Moreover, the morphism 
\[
\hrat(X) \from \hrat(Y)
\]
induced by a morphism of schemes
\[
X \to Y
\]
is strictly compatible with the monoid structures. 
\end{proof}

%%%%%%%%%%%%%%%%%%%%%%%
\section{Rational motivic path cogroupoid}%%%%%
%%%%%%%%%%%%%%%%%%%%%%%

\sProposition{pmon}{
The natural tensor product
\[
- \otimes_\one -
\]
of commutative monoid objects in a monoidal category makes $\Mdga(Z,\QQ)$ into a monoidal model category. (The subscript $\one$ refers to the unit object $\one = \QQ(0)$ of $\Mdga$, and plays the role of base-ring.) In particular, the homotopy category $\DMdga(Z,\QQ)$ has an induced monoidal structure. We denote the induced tensor product by
\[
- \otimes_\one^\LL -.
\]
}

\segment{1611290926}{Definition}%%%%%%
Fix a cofibrant replacement
\[
\hrat(X)^c \to \hrat(X)
\]
and recall that we've fixed a fibrant replacement
\[
\one \to \onef.
\]
We consider a fixed integral point $x \in X(Z)$. This gives rise to a morphism of monoid objects
\[
x^\lor: \hrat(X)^c \to \hrat(X) \xto{x^*} \one \to \onef.
\]
We define the \emph{coordinate ring of rational motivic loops}, or just \emph{loop ring} at $x$, by
\[
B_x := x^\lor
 \underset{\hrat(X)^c} {\overset{ho} \amalg}
  x^\lor.
\]
More generally, given $\xi, \eta \in \Hom_{\DMdga}(\hrat(X), \one)$ we let $\tilde\xi$, $\tilde\eta$ denote representatives in
\[
\Hom_{\Mdga}(\hrat(X)^c, \onef)
\]
and we define the \emph{coordinate ring of rational motivic paths}, or just \emph{path ring} from $\xi$ to $\eta$, by 
\[
{_\eta B_\xi}: = \tilde \eta 
\underset{\hrat(X)^c} {\overset{ho} \amalg}
\tilde \xi.
\]

\sLemma{1611290927}{
The path ring $_\eta B_\xi$ is independent of the choices of representatives $\tilde \eta$, $\tilde \xi$ up to canonical isomorphism in $\DMdga(Z,\QQ)$.
}

\begin{proof}
Suppose $\tilde \eta': \hrat(X)^c \to \onef$ is homotopy equivalent to $\tilde \eta$, and let $H: C(X)^c \times I \to \onef$ be a homotopy. Then $H$ induces morphisms of diagrams 
\[
\xymatrix{
\big( 
\onef \xfrom{\tilde \eta} \hrat(X)^c \xto{\tilde \xi} \onef
\big)
\ar[d] \\
\big(
\onef \xfrom{H} \hrat(X)^c \times I \xto{\tilde \xi} \onef 
\big)
\\
\big(
\onef \xfrom{\tilde \eta'} \hrat(X)^c \xto{\tilde \xi} \onef
\big)
\ar[u] 
}
\]
which are object-wise weak equivalences.
\end{proof}

\segment{1611290928}{}
We endow the family 
\[
\left\{ {_\eta B_\xi} \right\}_{\xi, \eta}
\]
of objects of $\DMdga(Z,\QQ)$ with a cogroupoid structure as follows. The counit of $_\xi B_\xi$ is the morphism
\[
_\xi B_\xi \to \one
\]
represented by the obvious maps
\[
_\xi B_\xi \to \onef \from \one
\]
in $\Mdga(Z,\QQ)$. For the path cocomposition morphisms, we write
$
H := \hrat(X)^c
$
and consider three morphisms
\[
\ze, \eta, \xi: H \tototo \one^f.
\]
We then have a morphism of diagrams
\[
\xymatrix{
&& H \ar[r]^\ze \ar@{=}[d] & \one^f 
&&
&& H \ar[r]^\ze \ar[d]^\eta & \one^f
\\
& H \ar@{=}[r] \ar@{=}[d] & H & 
\ar[rr] &  &
& \one \ar[r] \ar[d] & \one^f 
\\
H \ar@{=}[r] \ar[d]_\xi & H &&
&&
H \ar[r]^\eta \ar[d]^\xi  & \one^f 
\\
\one^f &&&
&& 
\one^f
}
\]
which gives rise to
\[
{_\ze B_\xi} \to
\left( \ze \coprod_H^{ho} \eta \right)
\; \coprod^{ho} \;
\left( \eta \coprod_H^{ho} \xi \right)
= {_\ze B_\eta} \bigotimes_\one^{\LL}
{_\eta B_\xi}.
\]
We omit the verification of associativity as well as the torsor condition.

%%%%%%%%%%
\section{Cellular motives and cellular motivic dga's}%%%%%%
\label{tstrsec}
%%%%%%%%%%%%%

\segment{celmot}{Cellular motives}

Under the assumption that $Z$ obeys Beilinson-Soul\'e vanishing, we discuss the category 
\[
\DAcell(Z,\QQ) \subset \DA(Z,\QQ)
\]
of \emph{cellular motivic complexes}, and its relationship to the Tannakian category of mixed Tate motives constructed by Levine \cite{LevineVanishing}. A similar discussion may be found in Iwanari \cite{Iwanari}.

\ssProposition{extstr}{
Let $M$ be a stable monoidal model category, $T^c$ a tensor-triangulated subcategory of $D:=\Ho M$ with a $t$-structure with heart $A^c$. Let $N^c$ be the full subcategory of $M$ whose objects are those of $T^c$. Assume that the objects of $N^c$ are compact. Let $N$ be the closure of $N^c$ under small homotopy colimits and desuspensions. Then $N$ is stable monoidal, the t-structure on $T^c$ extends to $T:= \Ho N$, and if $A$ denotes the heart of this t-structure, then every object of $A$ is a small colimit of objects of $A^c$.
}

\ssRemark{}{
The situation described in the proposition is summarized by the following diagram,
\[
\xymatrix{
N^c \ar[d] \ar@{}[r]|\subset & N \ar[d] \ar@{}[r]|\subset & M \ar[d]  \\
T^c \ar@{}[r]|\subset & T \ar@{}[r]|\subset & D \\
A^c \ar@{}[u]|\cup \ar@{}[r]|\subset & A \ar@{}[u]|\cup
}
\]
with input $M$ and  $T^c$, and output $N$ and $A$.
}

\begin{proof}
We refer to the definition of t-structure given in Lurie \cite[Definition 1.2.1.1]{LurieAlg}. Let $N^c_{\ge 0}$ denote the full subcategory of $N^c$ corresponding to $T^c_{\ge 0}$, and similarly for $N^c_{\le 0}$. Let $N_{\ge 0}$ denote the closure of $N^c_{\ge 0}$ under small filtered homotopy colimits, and similarly for $N_{\le 0}$. Let $T_{\ge 0}$ denote the homotopy category of $N_{\ge 0}$, and similarly for $T_{\le 0}$.

We have
$\Hom(X,Y) = 0$ for $X\in N_{\ge 0}$, $Y \in N_{\le -1}$. Indeed, 
\[
\Hom(\hocolim X_i, \hocolim Y_j) = \holim_i \hocolim_j \Hom(X_i, Y_j)
\]
by compactness. 

An arbitrary object $X = \hocolim X_i$ of $T$ fits into an exact triangle
\[
\tag{*}
X_{\ge 0 } \to X \to X_{\le -1}.
\]
Indeed, we have triangles like so:
\[
(X_i)_{\ge 0} \to X_i \to (X_i)_{\le -1}.
\]
Since homotopy colimits preserve exact triangles, we obtain (*) by taking homotopy colimits. 

The remaining properties of a t-structure are clear. 

For the statement about the heart, we resort to the language of infinity categories. The full subcategory $N_{\ge 0, \le 0}$ of $N$ is discrete, hence equivalent to the heart; moreover, homotopy colimits in $N_{\ge 0, \le 0}$ are just ordinary colimits. So it will suffice to show that 
\[
\Ind (N_{\ge 0, \le 0}) = \Ind(N_{\ge 0}) \cap \Ind(N_{\le 0}). 
\]
Suppose that $X$ is in the intersection. Write $X = \colim X_i$, $X_i \in N_{\ge 0}$. Consider the cofiber sequence 
\[
\colim \tau_{\ge 1}(X_i) \to 
X \to
\colim \tau_{\le 0 }X_i.
\]
The right hand object is in $\Ind(N_{\le 0})$, as is the middle object. Hence so is the left hand object. But the left hand object is also in $N_{\ge 1}$. Hence it is zero. Hence the map on the right is iso. 
\end{proof}

\ssegment{cellmot}{Definition}%%%%%
As mentioned above, we assume $Z$ obeys Beilinson-Soul\'e vanishing. We apply Proposition \ref{extstr} with
\[
M := \MM(Z,\QQ)
\]
and
\[
T^c := \DTM^c(Z,\QQ)
\]
the smallest tensor-triangulated subcategory of $\DA(Z,\QQ)$ containing the Tate objects. Under Beilinson-Soul\'e vanishing, Levine constructs a t-structure on $\DTM^c(Z, \QQ)$. As in the proposition, we consider the full subcategory
\[
\MM_\cell^c(Z,\QQ) \subset \MM(Z,\QQ)
\]
(denoted $N^c$ in the proposition) whose objects are in $\DTM^c(Z, \QQ)$, and we let $\MM_\cell(Z,\QQ)$ denote the closure of $\MM_\cell^c(Z,\QQ)$ under homotopy colimits. We define the category of \emph{cellular motives} by
\[
\DA_\cell(Z,\QQ) := \Ho \MM_\cell (Z,\QQ).
\]
The proposition then gives us a t-structure on $\DA_\cell(Z,\QQ)$ extending the one on $\DTM^c(Z,\QQ)$. Moreover, by the proposition, its heart $\TM(Z,\QQ)$ is equal to $\Ind \TM^c(Z,\QQ)$, the Ind-completion of the category $\TM^c(Z, \QQ)$ of (compact) mixed Tate motives. 

We summarize the above discussion in the following proposition.

\begin{ssprop}[Cellular motives]%%%%%%%
\label{cellmot}
Assume $Z$ obeys Beilinson-Soul\'e vanishing. Then Levine's t-structure on $\DTM^c(Z,\QQ)$ extends to a t-structure on $\DA_\cell(Z,\QQ)$. Every object of its heart $\TM(Z,\QQ)$ is a small colimit of objects of $\TM^c(Z,\QQ)$.
\end{ssprop}

%%%%%%%%
\segment{}{Cellular motivic dga's}%%%%%%%%%%%%
%%%%%%

\ssegment{13.1}{}
We define the \emph{category of cellular motivic dga's}
\[
\Mdga_\cell = \Mdga_\cell(Z,\QQ)
\]
to be the category of commutative monoid objects in $\MM_\cell(Z,\QQ)$.

\ssProposition{13.3}%%
{
The category $\Mdga_\cell(Z,\QQ)$ is closed under small homotopy colimits.
}

\begin{proof}
Special case of proposition \ref{13.4} below. 
\end{proof}

\segment{dophil}{Definition}%%%%
We say that a monoidal model category $M$ is \emph{monoidophilic} if the category $\cMon(M)$ of commutative monoids in $M$ inherits the structure of a model category with right-Quillen forgetful functor. 

\ssProposition{13.4}%%%%
{
Let $M$ be a monoidophilic model category. Let $T$ be a subcategory closed under weak equivalences, small homotopy colimits and derived tensor products.\footnote{The last condition (in the presence of the first condition) reduces to $T$ closed under (ordinary) tensor products of cofibrant objects.} Then $\cMon(T)$ is closed under small homotopy colimits. 
}

\begin{proof}
It's enough to check that $\cMon(T)$ is closed under small homotopy coproducts, realizations and filtered homotopy colimits. (Here ``realization'' means homotopy colimit over $\Delta^{\op}$.) Realizations and filtered homotopy colimits are both \emph{sifted}, hence, by Corollary 3.2.3.2 of Lurie \cite{LurieAlg}, compatible with the forgetful functor. Turning to homotopy coproducts in $\cMon(T)$, those are coproducts of derived tensor products in $T$. 
\end{proof}

%%%%%%%%%%%%%%%%
\section{Interplay between hopfoids and localization}%%%%%%%
%%%%%%%%%%%%%%%%%

\segment{15.1}{Definition}%%%%%
A \emph{(set-based) hopfoid} in a monoidal category $M$ consists of a set $S$, and a cogroupoid in $\cMon(M)$ with vertices in $S$. In a bit more detail, a family $\{_tB_s\}_{s,t \in S}$ of objects of $M$, a monoid structure
\begin{align*}
\one \To {_t B_s}, && {_t B_s^{\otimes 2}} \To {_t B_s}
\end{align*}
on each ${_tB_s}$, plus counits, cocompositions, and antipodes
\begin{align*}
{_sB_s} \to \one
&&
{_uB_t} \otimes {_tB_s} \from {_uB_s}
&&
{_tB_s} \to {_sB_t}
\end{align*}
such that the multiplication maps are unital and associative and the comultiplication maps are counital, coassociative, and invertible. A \emph{commutative hopfoid} is a hopfoid whose multiplication maps are commutative. As an example, $\Spec$ of a commutative hopfoid in $k$-modules is a groupoid in $k$-schemes in a set-based sense. A \emph{morphism of hopfoids}
\[
\{B\}_{s \in S} \to \{B'\}_{s' \in S'}
\]
consists of a map $f: S \to S'$ of sets, plus a family of morphisms
\[
\{ {_tB_s} \to {_{f(t)}B'_{f(s)}} \}_{s,t \in S}
\]
in $M$ commuting with the structure morphisms of $B$ and $B'$. We write $\cHfd(M)$ for the category of commutative hopfoids in $M$.

\segment{15.2}{}%%%%%%
If $\Mm$ is a monoidophilic model category (\ref{dophil})  and
\[
F_* : \cMon \Mm \to \Mm
\] 
is the forgetful functor, then the right derived functor
\[
RF_*: \Ho \cMon \Mm \to \Ho \Mm
\]
lifts to a functor
\[
\tag{$*$}
\Ho \cMon \Mm \to \cMon \Ho \Mm. 
\]
Given
\[
\one \xto{\epsilon} M,
\hspace{5ex}
M^{\otimes 2} \xto{\mu} M
\]
an object of $\Ho \cMon \Mm$, we define the associated commutative monoid in $\Ho \Mm$ as follows. For the unit we simply take $\Ho (\epsilon)$. For the counit, we choose a cofibrant replacement
\[
M^c \to M
\]
and we take the morphism
\[
M \otimes^\LL M \to M
\]
represented by the composit
\[
M^c \otimes M^c \to M\otimes M \xto{\mu} M.
\]
Verification of the axioms is tedious, but straightforward.

\begin{sprop}[Interplay between Hopfoids and localization]
\label{12.3}%%%%%%
The functor \ref{15.2}($*$) may be upgraded to a functor
\[
\tag{$*$}
\Cogpd \Ho \cMon \Mm \To \Cogpd \cMon \Ho \Mm
 = \cHfd \Ho \Mm.
\]
\end{sprop}

We omit the details.

%%%%%%%%%%%%%%%%
\section{Connectedness and concentration theorems}%
%%%%%%%%%%%%%%

\segment{tooth}{}%%%%%%
If $k$ is a field, then for us, a commutative dga over $k$ is a commutative monoid object of the category $\Cpx(k)$ of complexes of $k$-vector spaces. We impose no boundedness condition. We denote the category of commutative dga's by $\dga(k)$. We use cohomological notation, so the differential increases degrees. We endow $\Cpx(k)$ with the projective model structure, and we endow $\dga(k)$ with the associated model structure. Explicitly, this means that weak equivalences are defined to be quasi-isomorphisms and fibrations are defined to be surjections (of underlying sets). This makes $\dga(k)$ into a monoidal model category with cocartesian monoidal structure ($\otimes = \coprod$).

\segment{booth}{}%%%%%%
Before we turn to the proof of the connectedness and concentration theorems, we review the construction of explicit cofibrant replacements in $\dga(k)$ from Bousfield--Gugenheim \cite[\S4]{BousfieldGugenheim}, making slight modifications to accommodate dga's that are not concentrated in positive degrees and to suit our purposes. (The reader who isn't familiar with this material should nevertheless consult Bousfield--Gugenheim for details.)

\ssegment{sleuth}{}%%%%%%%%
The category $\dga(k)$ contains infinite coproducts, which may be constructed concretely as directed colimits of finite coproducts. For $n \in \ZZ$ we let $S(n)$ be the free graded-commutative algebra on one generator $a$ in degree $n$ with differential induced by
\[
da = 0.
\]
Again for $n\in \ZZ$, we let $T(n)$ be the free graded-commutative algebra on generators $b$ in degree $n$ and $c$ in degree $n+1$ with differential induced by
\[
db = c.
\] 

\ssegment{foof}{}%%%%%
We denote the degree of a homogeneous element $x$ by $|x|$. Consider a morphism of dga's $f:X\to Y$. In constructing a cofibrant replacement
\[
X \xto{\be} L \xto{\psi} Y
\]
of $f$ (meaning that $\be$ is a cofibration and $\psi$ is a trivial fibration), we may first factor $f$ into
\[
X \xto{\ka} X' \xto{f'} Y
\]
where $\ka$ is a cofibration and $f'$ is a fibration inducing surjections on cohomology
\[
HX' \surj HY.
\]
We recall how this is accomplished (although this will not intervene in the sequel). We choose a set $\Yy$ of homogeneous elements spanning $Y$ and a set $\Zz$ of homogeneous elements spanning $ZY$ and we define 
\[
X' = X \otimes 
\bigotimes_{y \in \Yy} T(|y|) \otimes
\bigotimes_{z \in \Zz} S(|z|)
\]
and we define $\ka$ and $f'$ in the obvious way. As explained in loc. cit. the fact that $\ka$ is a cofibration follows from the fact that the dga's $S(n)$, $T(n)$ are themselves cofibrant and the fact that cofibrations are closed under pushouts. The map $f'$ is surjective because the map
\[
\bigotimes_{y \in \Yy} T(|y|) 
\to
X' \to X
\]
is surjective, and it induces surjections on cohomology because the map
\[
\bigotimes_{z \in \Zz} S(|z|)
\to
X' \to X
\]
induces surjections on cohomology.\footnote{In fact, even if the spanning sets chosen above are bases, this construction remains highly redundant, and we will have to replace it by one that is not redundant in a special case below. }

\ssegment{roofer}{}%%%%%
Therefore, after possibly replacing $X$ by $X'$ we may assume $f$ is a cofibration which induces surjections on cohomology. Under this assumption we construct a factorization
\[
X \xto{\be'} L'_f \xto{\psi'} Y
\] 
where $\be'$ is a cofibration, $\psi'$ is a fibration and
\[
\ker H(\be') = \ker H(f).
\]
For each $n\in\ZZ$ we let $R^n$ denote the fibered product of vector spaces
%\[
%\xymatrix{
%R^n \ar[r] \ar[d] 
%&
%Z^nX \ar[d]
%\\
%Y^{n-1} \ar[r]
%&
%Z^nY,
%}
%\]
\[
R^n = Y^{n-1} \times_{Z^nY} Z^nX
\]
--- a set of pairs $(y,x)$ such that 
\[
dy = fx,
\]
and we choose a set $\Rr^n \subset R^n$ which spans $R^n$. We set 
\[
\Rr = \bigcup_{n \in \ZZ} \Rr^n.
\]
We then have a commuting square like so
\[
\xymatrix{
\underset {(y,x) \in R} \bigotimes S(|x|) \ar[r] \ar[d] &
X \ar@{.>}[d]^-{f} \\
\underset {(y,x) \in R} \bigotimes T(|y|) \ar@{.>}[r] &
Y
}
\]
and we let $L'_f$ be the pushout of the solid arrow diagram. That $\be'$ is a cofibration follows from the fact that the natural maps
\[
S(n) \to T(n-1)
\]
are cofibrations, and the remaining properties are easily verified. 

\ssegment{block}{}%%%%%%%%
We continue with the map $f:X \to Y$ under the assumptions of segment \ref{roofer}.
To accord with the notation in Bousfield--Gugenhem, we set  
\begin{align*}
L(2) := L'_f,
&&
\be_2 := \be',
&&
\psi_2 := \psi'
\end{align*}
and define recursively
\[
L(n+1) = L'_{\psi_n}
\xto{\psi_{n+1} := \psi'}
Y.
\]
Finally, we set 
\[
L := \colim_n L(n).
\]
The induced map $\be$ is a cofibration and the induced map $\psi$ is a trivial fibration, as is easily verified.

%rrr

\segment{14.1}{}%%%%
Let $Z$ be an integral scheme with function field $K$ of characteristic zero. In proving the connectedness and concentration theorems we will visit many of the categories discussed above: the motivic model category
\[
\MM = \MM(Z, \QQ)
\]
(\S\ref{bubble}), its homotopy category $\DA = \DA(Z,\QQ)$, the category of motivic dga's $\Mdga = \Mdga(Z, \QQ)$ (\S\ref{1611290924}), its homotopy category
\[
\DMdga = \DMdga(Z, \QQ),
\]
the category of commutative dga's over $\Qp$: $\dga = \dga(\Qp)$, its homotopy category $\Ddga = \Ddga(\Qp)$, as well as the (unbounded) derived category $D(\Qp)$ of $\Qp$-vector spaces.

\segment{J2}{}%%%%
There's a $p$-adic \'etale realization functor on the level of model categories which is covariant, monoidal, colimit preserving and t-exact on the cellular subcategory $\Mdga_\cell(Z,\QQ)$.\footnote{
Let us emphasize that while we do not assume $p$ to be invertible on $Z$, we only compute \'etale cohomology after passing to a field of characteristic zero.
} We sketch its construction in the language of infinity categories. Let $\QQ_{Z_\et}$ denote the constant sheaf on $Z_\et$, let $\Vect(\QQ_{Z_\et})$ denote the category of $\QQ_{Z_\et}$-vector spaces (that is, the category of sheaves of $\QQ$-vector spaces), and let $\Cpx \Vect(\QQ_{Z_\et})$ denote the category of complexes in $ \Vect(\QQ_{Z_\et})$.  We let $\DD A =\DD A(Z, \QQ)$ denote the infinity category obtained from 
\[
\Cpl \Vect(\QQ_{Z_\et}) 
\llbracket (\mbox{local quasi-isomorphisms})\inv \rrbracket
\]
by localization and stabilization as in the model categorical setting. Let $f: X \to Z$ denote the structure morphism of $X$. Let $ \Qp_{X_{\overline K, \et}}$ denote the constant sheaf on $X_{\overline K, \et}$ associated to $\Qp$.  The language of infinity categories provides a canonical lift
\[
 \RR f_* \Qp_{X_{\overline K, \et}} \in \DD(\Qp)
\]
of the derived pushforward. The functor 
\[
h_{\overline K, \et}: \Sm_{/Z}^{\op} \to \DD(\Qp)
\]
\[
 X \mapsto  \RR f_* \Qp_{X_{\overline K, \et}}
\]
satisfies \'etale descent and $\AA^1$-invariance, and is $(\PP^1, \infty)$-stable. It then follows from the universal mapping property of $\DD A$ that there's a symmetric monoidal functor
\[
\RRe_{\overline K, \et}: \DD A(Z,\QQ) \to \DD(\Qp)
\]
which factors $h_{\overline K, \et}$ up to homotopy (c.f. Robalo  \cite[Corollary 2.39]{RobaloBridge}). 

\segment{shu1}{}%%%%%%%%%%%%%%%%
We let $\DDMdga(Z,\QQ)$ denote the infinity category of commutative monoids in $\DD A(Z,\QQ)$ and we let $\DDdga(\Qp)$ denote the infinity category of commutative differential graded $\Qp$-algebras. According to Hinich \cite{HinichRectification},
\[
\DDMdga = \Mdga \llbracket W\inv \rrbracket
\]
and
\[
\DDdga = \dga \llbracket W\inv \rrbracket
\]
where, in both cases, $W$ denotes the class of weak equivalences. Cup product gives 
$\RR f_*\Qp_{X_{\overline K, \et}}$
the structure of a differential graded algebra, which we denote by 
\[
\CC^*_\et(X_{\overline K}, \Qp).
\]
It follows from the construction of the factorization of $h_{\overline K, \et}$ above (\S\ref{J2}) that the functor $\RRe_{\overline K, \et}$ preserves small (homotopy) colimits. It is then a formal consequence that $\RRe_{\overline K, \et}$ can be upgraded to a functor
\[
\RRe_{\overline K, \et}^a:\DDMdga(Z,\QQ) \to \DDdga(\Qp)
\]
which factors the functor
\[
\CC^*_\et \big( (\cdot)_{\overline K}, \Qp \big):
\Sm_{/Z}^{\op} \to 
\DDdga(\Qp)
\]
up to homotopy.

\segment{ja1}{}%%%%%%
We note for use in segment \ref{diaproof} below that for similar reasons, there's a further factorization of the unstructured realization functor $\RRe_{\overline K, \et}$ through a structured realization functor: let $\DD(\Qp_{K, \et})$ denote the derived infinity category of complexes of sheaves of $\Qp$ vector spaces on the \'etale site of $\Spec K$; then there's a functor
\[
\RRe_\et:
\DD A(Z,\QQ) \to 
\DD(\Qp_{K,\et})
\]
which factors the functor
\[
X \mapsto \RR f_* \Qp_{X_K}
\]
up to homotopy.

This structured realization functor can equally be upgraded to a realization functor on motivic dga's. Let $\DDdga(\Qp_{K_\et})$ denote the infinity category of sheaves of (for us always commutative) differential graded $\Qp_{K_\et}$-algebras. Then there's a functor
\[
\RRe_\et^a:
\DDMdga(Z,\QQ)
\to
\DDdga(\Qp_{K_\et})
\]
which factors the functor
\[
\CC^*_\et( \;\cdot\;, \Qp): \Sm^{\op}_{/Z} \to
 \DDdga(\Qp_{K_\et})
 \]
 which sends a smooth scheme $X$ to the derived pushforward of its generic fiber 
 \[
 \CC^*_\et( X_K, \Qp) := \RR f_* \Qp_{{X_K}_\et}
 \]
 equipped with cup product.

%%%%%%%%%%%%%%%%%%%%%%%%%%%%% 
 \segment{1611290929}{Poof of the connectedness and concentration theorems (\ref{ctd} and \ref{ctr})}%%%%%

\ssegment{shu3}{}%%%%%%%%%%%%%%
 The forgetful functor
\[
F_*: \Mdga \to \MM
\]
induces a functor
\[
RF_*: \DMdga \to \DA.
\]
Let $I$ denote the category
\[
\bullet \from \bullet \to \bullet.
\]
Recall that if $M$ is a model category, then the homotopy pushout is by definition a functor 
\[
\Ho(M^I) \to \Ho(M)
\]
from the homotopy category of the diagram category $M^I$. Recall that $\DA_\cell = \DA_\cell(Z,\QQ)$ denotes the category of cellular motives (\ref{cellmot}), and that its heart $\TM = \TM(Z,\QQ)$ (for the t-structure discussed in the same segment) is equal to $\Ind \TM^c = 
\Ind \TM^c(Z,\QQ)$, where $\TM^c$ is the \textit{usual} category of mixed Tate motives.

With these notations, we have (for each $i \in \ZZ$) a commuting diagram like so.
\[
\xymatrix{
%1
&
\Ho(\Mdga^I) \ar[r]^-{Re_{\overline K, \et}^{a, I}} \ar[d]_{\hocolim} &
\Ho (\dga^I)  \ar[d]^{\hocolim}
\\
%2
&
\DMdga \ar[r]^-{Re^a_{\overline K, \et}} \ar[d] &
\Ddga  \ar[d]^{RF_*}
\\
%3
\DA_\cell \ar@{}[r]|\subset \ar[d]_{H_i^t}  &
\DA \ar[r]^-{Re_{\overline K, \et}} &
D(\Qp) \ar[d]^{H_i}
\\
%4
\TM \ar[rr]_{Re^\m{\heartsuit}_{\overline K, \et}} &
&
\Vect (\Qp).
}
\]
Here $\Vect \Qp$ denotes the category of \emph{all} vector spaces. The lower-left portion of the diagram is available only under the assumption that a mixed Tate t-structure exits.

%ggggg

\begin{sslm}%%%%%%%%
\label{cons}
If an object $A$ of $\TM(Z,\QQ) = \Ind \TM^c(Z,\QQ)$ goes to zero under $Re^\m{\heartsuit}_{\overline K, \et}$ then $A=0$. 
\end{sslm}

\begin{proof}
Write 
\[
A = \colim A_\bullet
\]
\[
A_\bullet : I \to \TM^c(Z,\QQ)
\]
as a filtered colimit of compact objects. The functor $Re_{\overline K, \et}^\m{\heartsuit}$ commutes with colimits since it's a left adjoint, so $A$ maps to
\[
\colim Re_{\overline K, \et}^\m{\heartsuit, c} A_i.
\]
The vanishing of a filtered colimit of compact objects means that there exists a cofinal subdiagram $F: J \to I$ such that for any morphism
$
f
$
in $J$, 
\[
Re_{\overline K, \et}^\m{\heartsuit} (A_\bullet(F(f))) = 0.
\]
However, the restriction of $Re_{\overline K, \et}^\m{\heartsuit}$ to $\TM^c$ is fully faithful, so that in this case we have
\[
A_\bullet(F(f)) = 0.
\]
Hence 
\[
A = \colim_I A_\bullet = \colim_J A_\bullet = 0. \qedhere
\]
\end{proof}

\ssegment{}{}%%%%%%
Evidently, to establish the non-\textit{mixed Tate} case, we must show:
\begin{itemize}
\item[(*)]
 $H_i RF_* Re^a_{\overline K, \et} ({_\eta B_\om}) $ is nonzero for $i=0$ and is zero for $i \neq 0$. 
 \end{itemize}
 According to lemma \ref{cons}, statement (*) establishes the mixed Tate case as well. The remainder of the proof is therefor the same in both mixed Tate and non-\textit{mixed Tate} cases.

\ssegment{fal10}{}%%%%%
We fix cofibrant / fibrant replacements $QX \xto{c} CX$, $\one \xto{f} P\one$ in $\Mdga(Z,\QQ)$, and suppose $\om$, $\eta$ are represented by morphisms
\[
P\one
\xfrom{\eta}
QX \xto{\om} P\one.
\]
We switch to working in the infinity category $\DDMdga(Z,\QQ)$ associated to $\Mdga(Z,\QQ)$. There we may pick inverses up to homotopy 
\[
QX \xfrom{c\inv} CX
\]
\[
\one \xfrom{f\inv} P\one.
\]
We let $\om' : = f\inv \om c\inv$, $\eta' = f\inv \eta c\inv$.
We then have an equivalence of diagrams
\[
( P\one \xfrom{\eta}
QX \xto{\om} P\one)
 \sim 
( \one \xfrom{\eta'}
CX \xto{\om'} \one)
\]
in $\DDMdga(Z,\QQ)$\footnote{i.e. a morphism in the diagram category which induces an isomorphism in the homotopy category; see, for instance, Proposition 1.2.7.3 of Lurie \cite{LurieTopos} for the formation of the category of diagrams.}, and hence an equivalence of pushouts 
\[
P\one \coprod_{QX} P\one \sim \one \coprod_{CX} \one.
\]
The functor 
\[
\RR e_{\overline K, \et}^a: 
\DDMdga(Z,\QQ) \to
 \DDdga(\Qp)
\]
preserves equivalences\footnote{c.f. Proposition 1.2.3.1 of Lurie \cite{LurieTopos}} and respects pushouts and terminal objects. 
Indeed, the functor
\[
\RR e_{\overline K, \et}: \DD A(Z,\QQ) \to \DD(\Qp)
\]
from the infinity category associated to the motivic model category to the derived infinity category of complexes of $\Qp$-vector spaces preserves colimits and is monoidal. Since the forgetful functors
\[
\DDMdga(Z,\QQ) \to \DD A(Z,\QQ)
\]
and
\[
\DDdga(\Qp) \to \DD(\Qp)
\]
are conservative and preserve sifted colimits, it follows that the functor of monoid objects $\RR e_{\overline K, \et}^a$ preserves sifted colimits and finite products, hence all small colimits.
 
Hence, we have an equivalence
\[
\RR \m{e}_{\overline K, \et}^a (P\one \coprod_{QX} P\one)
=
\Qp \coprod_{C_\et^*(X_{\overline K})} \Qp.
\]

\ssegment{derdia}{}%%%%
We claim that the derived augmentations
\[
\tag{*}
\Qp
\xfrom{\RR\m{e}^a_{\overline K, \et} \eta'}
C^*_\et(X_{\overline K}, \Qp)
 \xto
 {\RR\m{e}^a_{\overline K, \et} \om'}
 \Qp
\]
are in the image of the map
\[
\Hom_{\dga}(C_\et^*(X_{\overline K}), \Qp)
 \to
 \pi_0 \Hom_{\DDdga}(C_\et^*(X_{\overline K}), \Qp). 
\]
In fact, we make the stronger claim that the target is a singleton. To prove this fact we will construct a particularly nice cofibrant replacement. While we're at it, we will also construct a formal model which will serve us below.

\ssegment{pipep}{Cofibrant model and formal model}%%%%%%
Since $X$ is an affine curve, the dga $C^*_\et = C^*_\et(X_{\overline K}, \Qp)$ is cohomologically concentrated in degrees $0$ and $1$; recall that we're also assuming that the cohomology in degree zero is equal to $\Qp$. We choose a vector subspace $M$ of $Z^1C^*_\et$ which splits the projection $Z^1C^*_\et \surj H^1C^*_\et$ and we let 
\[
S(M) = \Qp \oplus M \oplus \wedge^2M \oplus \cdots
\]
denote the free graded-commutative algebra generated by $M[-1]$ with zero differential.\footnote{The letter $S$ constitutes a notational pun, as it stands for both \textit{symmetric algebra} and for \textit{sphere object}.}
 The natural map of dga's 
\[
\tag{P}
S(M) \to C^*_\et
\]
is both surjective and induces surjections on cohomology. Hence, we may apply the construction of segments \ref{roofer}-\ref{block} to obtain a cofibrant replacement
\[
S(M) \to L \to C^*_\et
\]
of the morphism (P). Since $S(M)$ is cofibrant, $L$ itself is a cofibrant model of $C^*_\et$. It follows from the construction that 
\begin{itemize}
\item[(L1)]
$L^i = 0$ for $i<0$, and
\item[(L2)]
$L^0 = \Qp$.
\end{itemize}

Now let $J^1$ be any linear complement of the image of $M$ inside $L^1$ and extend $J^1$ to a graded vector subspace $J$ of $L$ by setting
\[
J^i = 
\left\{
\begin{matrix}
0 & i \le 0 \\
L^i & i \ge 2.
\end{matrix}
\right.
\]
Then evidently, $J$ is both a graded ideal and a sub-complex. Hence, taking the quotient by $J$, we obtain a map of dga's
\[
L \to \Qp[0] \oplus M[-1]
\] 
which is a quasi-isomorphism. This gives us our formal model
\[
'C := \Qp[0] \oplus M[-1].
\]

\ssegment{inorder}{}%%%%%%%%%
In order to compute 
$
\pi_0\Hom_{\DDdga(\Qp)}(C_\et^*, \Qp)
$
we use the cofibrant model $L$. The properties \ref{pipep}(L1--L2) guarantee that
\[
\Hom_{\dga}(L, \Qp) = \{*\}
\]
is a singleton. Since the latter surjects onto the derived hom
\[
\Hom_{\Ddga} (L, \Qp) = \pi_0 \Hom_{\DDdga}(L, \Qp),
\]
the claim (\ref{derdia}) follows.

\ssegment{1611290931}{}%%%%%%
We now switch to using the formal model $'C$ (\ref{pipep}). The derived augmentation appearing in diagram \ref{derdia}(*) gives rise equally to derived augmentations 
\[
\Qp
\from
{'C} \to \Qp.
\]
Our computation of the set of augmentations (\ref{inorder}) shows that this diagram lifts to the model category $\dga(\Qp)$. The colimit 
\[
\Qp \coprod_{C^*_\et(X_{\overline K})}^{\DDdga} \Qp
\]
of \ref{derdia}(*) is thus represented by the reduced bar complex. The latter is, by definition, the total complex of a simplicial dga which has ${'C}^{\otimes n}$ in degree $n$; the degeneracies are given by all maps
\[
( {'C})^{\otimes (n-1)} \to
 ({'C})^{\otimes n}
 \]
induced by the inclusion $\Qp \to {'C}$. By lemma 8.3.7 of Weibel \cite{Weibel}, we may mod out by degeneracies prior to forming the total complex. The result is $(M[-1])^{\otimes n}$ in degree $n$. Upon passing to the total complex, we obtain the tensor algebra on $M$ in degree $0$, and zero in all other degrees. This establishes the connectedness and concentration theorems.

\segment{1611290932}{Remark}%%%%%
In the course of the proof, we chose a splitting of the surjection 
\[
C_\et^1(X_{\overline K}, \Qp) \surj
H^1_\et(X_{\overline K}, \Qp)
\]
and denoted its image simply by $M\subset C_\et^1$. We noted that $\Qp[0] \oplus M[-1]$ has a quasi-isomorphism to $C^*_\et$ on the one hand, and that the former has only one augmentation on the other hand. This uniqueness of augmentations is reflected in the existence of compatible unipotent \'etale paths connecting every two base-points, depending on the choice of the splitting. Its de Rham cousin $\Om^*$ can be made more concrete. The resulting de Rham paths were observed by Deligne \cite{Deligne89} in a special case (under the further assumption that $X$ has a compactification $\overline X$ with $H^1(\overline X, \Oo) = 0$, in which case the resulting paths are in fact independent of the choice).

 We sketch a Tannakian construction of such paths which developed in conversations with Stefan Wewers.  In this setting, $M$ represents the choice of a subspace of $\Om^1(X)$ which maps isomorphically onto $H^1_\dR(X)$. Consider $(E, \nabla)$, a unipotent vector bundle with integrable connection. Since 
\[
H^1(X, \Oo) = 0,
\]
the underlying vector bundle $E$ is trivial. One can check that there's a unique $K$-form $V$ of $E$ such that $\nabla$ restricts to a $K$-linear map
\[
V \to V \otimes M.
\]
The functor 
\[
\om_M: E \mapsto V
\]
is a fiber functor. If $x$ is a $K$-point of $X$, then there are canonical isomorphisms
\[
V \xto{\sim} \om_x(E):=E(x),
\]
which, as $E$ varies, form an isomorphism of fiber functors
\[
\om_M \xto{\sim} \om_x.
\]

\segment{1611281015}{Remark}%year month day hr min
The infinity-categorical steps in our proof can be unravelled into (somewhat cumbersome) model-categorical ones. If one attempts to work naively with model categories however, it's hard to see any direct relation between the left-derived realization on the level of $I$-diagrams
\[
Re^{a, I}_{\overline K, \et} 
( P\one \from QX \to P\one) 
\in
 \Ho(\dga^I)
\]
(which a priori may require a fibrant replacement of the diagram), with a diagram of dga's of the form
\[
\Qp \from
C_\et^*(X_{\overline K})
\to 
\Qp
\]
to which we can apply our explicit computation via the classical bar construction. 

%%%%%%%%%%%%%%
\section{The unipotent \'etale fundamental group: comparison of external and internal viewpoints}%
\label{compsec}
%%%%%%%%%%%%%%%

\segment{fal1}{}
We put ourselves in the setting of the introduction ($Z$, $X$, $x$, $p$). We let $Z_p$ denote $\Spec$ of the complete local ring at $p$, and we let $X_p$ and $x_p$ denote the base change of $X$ and of $x$ to $Z_p$. We let $\et$ denote the pro-\'etale topology of Bhatt--Scholze \cite{ScholzeProet}. We let $K_p$ denote the function field of $Z_p$ and we fix an algebraic closure $\overline K_p$. If $f$ denotes the structure morphism
\[
f:X_p \to Z_p, 
\]
then $Rf_*\Qp_{X_{p,\et}}$ has the structure of a sheaf of a differential graded $\Qp_{Z_\et}$-algebras equipped with an augmentation induced by the section $x_p$. 

\segment{fal2}{}
We consider the model structure on the category $\dga(\Qp_{Z_\et})$ of Cisinski-D\'eglise \cite{CDHomAlg}. This allows us to take a homotopy colimit in order to form the \emph{rational p-adic \'etale loop ring} 
\[
B^\et_{x_p} := \hocolim
\left(
\Qp_{Z_{p,\et}}
\xfrom {x_p^*}
Rf_*\Qp_{X_{p,\et}}
 \xto{x_p^*}
 \Qp_{Z_{p,\et}} 
 \right).
\]
More generally, as in the motivic setting, if
\[
\Qp_{Z_{p,\et}}
\xfrom{\eta}
Rf_*\Qp_{X_{p,\et}}
 \xto{\xi}
 \Qp_{Z_{p,\et}} 
\]
are any two augmentations, the we may form the \emph{rational p-adic \'etale path ring}
\[
{_\eta B^\et_\xi} := \hocolim
\left(
\Qp_{Z_{p,\et}}
\xfrom{\eta}
Rf_*\Qp_{X_{p,\et}}
 \xto{\xi}
 \Qp_{Z_{p,\et}} 
 \right).
\]
With these definitions the family $\{ {_\eta B^\et_\xi} \}_{\xi,\eta}$ forms a cogroupoid in $\Ddga(\Qp_{Z_{p, \et}})$. Each $\Qp_{Z_{p, \et}}$-algebra $H^0({_\eta B^\et_\xi})$ is locally constant, so (by base-change to $\overline K_p$) the family $\{\Spec H^0({_\eta B^\et_\xi}) \}_{\xi,\eta}$ may equally be regarded as a groupoid in Galois representations.
 
 \segment{fal3}{}%%%%%%
In parallel with the motivic setting, the rational $p$-adic \'etale path rings of segment \ref{fal2} give rise to a sequence of maps $f_1^\et$, $f_2^\et$, $f_3^\et$, as in the diagram below
\[
\xymatrix{
X(Z_p) \ar[d]_-{f_1^\et} \ar@/^5ex/[dddr]^-{\mu}
\\
\Hom_{\Ddga(\Qp_{Z_\et})} (Rf_*\Qp_{X_\et}, \Qp_{Z_\et}) 
\ar[d]_-{f_2^\et}
\\
\Cotors B_{x_p}^\et
\ar[d]_-{f_3^\et}
\\
H^1\big(G_{K_p}, \Spec (H^0B^\et_{x_p})_{\overline K_p} \big)
\ar@{.>}[r]_-\sim
&
H^1(G_{K_p}, \Aut^\otimes \om_{x_p}^\un ),
}
\]
with
$
f_1^\et (y)
$
equal to the augmentation 
\[
Rf_*\Qp_{X_\et}
\xto{y^*}
 \Qp_{Z_\et}
\]
induced by $y$, 
\[
f_2^\et(\xi) := {_\xi B_{x_p}^\et},
\]
and
\[
f_3^\et(P) := \Spec H^0(P)_{\overline K}.
\]
We define the map $\mu$ in the diagram by sending a $Z_p$-point $y$ to the Tannakian path torsor 
\[
 \Isom^\otimes(\om_{x_p}^\un, \om_y^\un)
\]
through the category of unipotent $p$-adic \'etale local systems on $X_{\overline K_p}$. A $G_{K_p}$-equivariant isomorphism
\[
\Spec (H^0B^\et_{x_p})_{\overline K_p}
=
\Aut^\otimes \om_{x_p}^\un
\]
Will give rise to the dotted-arrow isomorphism. More generally, a $G_{K_p}$-equivariant isomorphism
\[
\Spec (H^0{_yB^\et_{x_p}})_{\overline K_p}
=
\Isom^\otimes(\om_{x_p}^\un, \om_y^\un)
\]
for each $y \in X(Z_p)$ will show the commutativity of the diagram. Our goal for this section is to establish these isomorphisms. We begin in segment \ref{ja2} in a general setting. 

\segment{ja2}{The unipotent fundamental group of a topos}
Let $X$ be a topos such that $\RR \Gamma(\QQ_X)$ is coconnective and let $x, y$ be points. We can associate to $x$ and $y$  a diagram
\[
\QQ \xfrom{y^*}
\RR \Gamma (\QQ_X) \xto{x^*} \QQ
\]
in the model (or infinity) category of (always commutative) dga's. (The model category structure follows, for instance, by combining Cisinski--D\'eglise \cite{CDHomAlg} with White \cite{White} as we did for $\Mdga$.) We then let $_yB_x$ be the homotopy colimit; its cohomology $H^0(_yB_x)$ is a $\QQ$-algebra. Let $\m{Uni}(\QQ_X)$ denote the category of unipotent local systems, and let $\om_x^\un$, $\om_y^\un$ denote the fiber functors
\[
\m{Uni}(\QQ_X) \xyto{}{} \Vect \QQ
\]
associated to $x$ and to $y$. Then there's an isomorphism
\[
\Spec H^0({_yB_x}) = \Isom^\otimes(\om_x^\un, \om_y^\un)
\]
(compatible with groupoid structures as we vary $x$ and $y$). Our goal here is to prove this statement, amplified by a certain base-changing property in a relative setting. 

\ssegment{21a}{}%%
For $X$ a topos and $K$ a field, we write $\Cpl(K_X)$ for the model category of complexes of $K_X$-modules, and $D(K_X)$ for its homotopy category, the (unbounded) derived category. We write $\dga(K_X)$ for the model category of commutative differential graded $K_X$-algebras, and $\Ddga(K_X)$ for its homotopy category.

\ssegment{21b}{}%%%%%
If $f:X \to Z$ is a morphism of topoi, then there's a commutative differential graded $K_Z$-algebra $C(X) = C_Z(X,K)$, the \emph{cochain complex of $X/Z$}, which is contravariant in topoi over $Z$, and whose underlying complex of $K_Z$-modules represents $Rf_*K_Z$. In one of two possible constructions, we note that $f_*$, being right adjoint to a monoidal functor, is lax-monoidal, hence sends algebras to algebras. This means, for instance, that for $E,F \in \Cpl(K_X)$, there's a morphism
\[
f_*E \otimes f_*F \to f_*(E \otimes F).
\]
Fixing a fibrant-replacement functor on $\dga(K_X)$, we let $P_X$ denote the fibrant replacement of the unit object $K_X$. We then set 
\[
C(X) := f_*P_X.
\]
If $m$ denotes the multiplication map
\[
P_X \otimes P_X \to P_X,
\]
then the product structure on $C(X)$, for instance, is given by the composite
\[
f_*P_X \otimes f_*P_X \to f_*(P_X \otimes P_X) \xto{f_*m} P_X.
\]
The dga $C(X)$ may also be interpreted as the endomorphism ring of the unit object $K_X$, a construction which is a bit awkward on the model categorical level, but automatic in the infinity-categorical setting.  

\ssegment{21c}{}%%%%
Assume given in addition to the morphism $f: X\to Z$ of topoi, a pair of sections $\xi, \eta : X \leftleftarrows Z$. There is then a diagram
\[
K_Z
\xfrom{\eta}
C(X) \xto{\xi} K_Z
\]
in $\dga(K_Z)$. We let $_\eta B_\xi$ denote its homotopy colimit. When $\eta = \xi$ we write simply $B_\xi$. We call $_\eta B_\xi$ the \emph{path ring of $X/Z$ from $\xi$ to $\eta$}, and we call $B_\xi$ the \emph{loop ring of $X/Z$ at $\xi$}.  As the sections $\xi$, $\eta$ vary, the objects $_\eta B_\xi$ form a commutative Hopfoid (\ref{15.1}) in $D(K_Z)$. Under the assumption that $C(X)$ is coconnective we may then take cohomology with respect to the canonical t-structure to obtain a commutative Hopfoid $\{H^0({_\eta B_\xi }) \}_{\xi, \eta}$ in the  category 
\[
\Vect K_Z
\]
of $K_Z$-modules.

\ssegment{21e}{}%%%%%
If $Z$ is a topos and $K$ a field, we let 
\[
D_\cell(K_Z) \subset D(K_Z)
\]
denote the smallest tensor-triangulated subcategory closed under small direct sums and containing the constant sheaf $K_Z$. Its heart, which we denote by $\Cell(K_Z)$ is equal to the Ind-completion of the category $\Uni(K_Z)$ of unipotent local systems.

\begin{ssthm}[Unipotent fundamental group of pointed topos]
\label{21d}%%%%%%
Let $K$ be a field containing $\QQ$. Consider a diagram of topoi
\[
\xymatrix{
X_z \ar[d]^{f_z} \ar[r]^\ze &
X \ar[d]_f 
\\
\Set \ar[r]_z \ar@/^7pt/[u]^x \ar@/^21pt/[u]^y &
Z \ar@/_7pt/[u]_\xi \ar@/_21pt/[u]_\eta
}
\]
in which all upwards arrows are sections of the respective downwards arrows, and all three squares commute. Assume that
$
z^*C(X) = C(X_z).
$
Assume that $C(X)$ is coconnective. Assume that every locally constant sheaf on $X_z$ which vanishes at $x$ vanishes globally. Let $\om_x^\un$, $\om_y^\un$ denote the fiber functors 
\[
\Uni(K_{X_z}) \xyto{}{} \Vect^c(K)
\]
induced by $x$, $y$. (Here $\Vect^c$ denotes the category of finite dimensional vector spaces.) Then we have an isomorphism 
\[
\tag{*}
\Spec z^* H^0({ _\eta B_\xi}) =
 \Isom^\otimes(\om_x^\un, \om_y^\un)
\]
which respects the groupoid structure as the sections $\xi$, $\eta$, $x$, $y$ vary.
\end{ssthm}

We begin by recalling two parts of Tannaka duality. 

\ssegment{hum1}{}%%%
Suppose $T$ is a Tannakian category over a field $k$ and $\om$ is a $k$-rational fiber functor. Then the fundamental group 
\[
\pi_1(X, \om) := \Aut^\otimes \om
\]
is \emph{characterized} as the unique proalgebraic affine $k$-group whose representation category fits into a commutative triangle
\[
\xymatrix{
\Rep^c_k \big( \pi_1(T,\om) \big)
 \ar[rr]^-\sim \ar[dr] &&
T \ar[dl]^\om
\\
&
\Vect^c k.
}
\]
(The superscript $c$ refers to \emph{compact objects}.) 

\ssegment{hum2}{}%%% 
If $\eta$ is another fiber functor, then the Tannakian path torsor 
\[
_\eta P_\om (X) := \Isom^\otimes(\om, \eta) 
\]
is characterized as the unique $\pi_1(T,\om)$--$\pi_1(T,\eta)$--bitorsor such that pushout along $_\eta P_\om$ fits into a commuting triangle
\[
\xymatrix{
\Rep^c_k \big( \pi_1(T,\om) \big) \ar[dr] \ar[dd]
\\
& T. \\
\Rep^c_k \big( \pi_1(T,\eta) \big) \ar[ur]
}
\]

\ssegment{Morita on Friday}{}%%%%%%
If $A$ is a dga, we let $\DMod A$ denote the derived category of $A$-modules. If $Y$ is a topos (regarded as topos over $\Set$), then Morita theory, as formulated by Lurie \cite[\S7.1.2.1--7]{LurieAlg} provides an equivalence of triangulated categories
\[
D_\cell(K_Y) = \DMod C(Y,K).
\]
Moreover, this equivalence is functorial in the following way. If 
\[
f: Y' \to Y
\]
is a morphism of topoi, then the induced morphism of dga's
\[
CY' \from CY
\]
gives rise, via the bar construction, to a morphism $f^*$ of triangulated categories, and the square
\[
\xymatrix{
\DMod CY' \ar@{=}[d] &
\DMod CY \ar[l]_{f^*} \ar@{=}[d]\\
D_\cell K_{Y'}  &
D_\cell K_{Y} \ar[l]
}
\]
commutes.

\ssegment{Koszul on Friday}{}%%%%%%
If $A$ is a differential graded coalgebra, we let $\DcoComod A$ denote the coderived category of comodules of Positselski \cite[\S4.4]{Positselski}, a certain variant of the derived category, which comes equipped with a natural t-structure whose heart is the same as the heart of the derived category of comodules. If $Y$ is a topos and $y: \Set \to Y$ is a point, then the Koszul duality of Positselski \cite[\S6.3]{Positselski} provides an equivalence of triangulated categories
\[
\DMod CY = \DcoComod B_y.
\]
Here $CY$ denotes the cochain complex of $Y/\Set$ (\ref{21b}) and $B_y$ denotes the loop ring of $Y/\Set$ at $y$ (\ref{21c}). Moreover, this equivalence is functorial in the following way. If
\[
f: Y' \to Y
\]
is a morphism of topoi and $y'$ is a point of $Y$ lying above $y$, then the a natural morphism of comonoids in dga's
\[
f_*: B_{y'} \to B_y
\]
gives rise, by forgetting along $f_*$, to a morphism $f^*$ of triangulated categories, and the square
\[
\xymatrix{
\DcoComod B_{y'} &
\DcoComod B_{y} \ar[l]_{f^*} \\
\DMod CY' \ar@{=}[u] &
\DMod CY \ar[l] \ar@{=}[u] \ar[l]
}
\]
commutes.

\ssegment{applying}{}%%%
Applying segments \ref{Morita on Friday} and \ref{Koszul on Friday} with $Y = X_z$, and $Y' = \Set$, we obtain a commuting triangle of triangulated categories
\[
\xymatrix
{
D_\cell K_{X_z} 
\ar@{=}[rr] \ar[dr]_{x^*}
&&
\DcoComod B_x \ar[dl]^{x^*}
\\
&
D(K).
}
\] 
On both sides, the natural t-structure is the one induced by the natural t-structure on $D(K)$. In the case of $\DcoComod B_x$, this is clear; in the case of $D_\cell K_{X_z}$, this follows from the connectivity property included in our assumptions, in view of the local-constancy of cohomology sheaves in the cellular category. It follows that the equivalence of triangulated categories is compatible with t-structures, and that the induced equivalence of hearts is compatible with the fiber functors. 

\ssegment{21e}{}%%%%%
Denoting the heart of the natural t-structure by $\heartsuit$, we have 
\[
 \Comod H^0z^* B_\xi  = \heartsuit \Comod z^*B_\xi
\]
because of flatness of modules over a field. We also have
\[
z^*B_\xi = B_x
\]
because of the base-changing property assumed in the statement. Combining these observations with the Morita theory and Koszul duality discussed above, we obtain a sequence of equivalences of neutralized Ind-Tannakian categories, which we may summarize as follows:
\begin{align*}
\Comod z^*H^0B_\xi
& = \Comod H^0z^* B_\xi 
\\
& \overset{\mbox{flatness}} = 
\heartsuit \DcoComod z^*B_\xi
\\
& \overset{\mbox{base change}} = 
 \heartsuit \DcoComod B_{x}
\\
& \overset{\mbox{Koszul duality}} =
  \heartsuit \DMod \big( C(X_z) \big)
\\
& \overset{\mbox{Morita theory}} =
  \heartsuit D_\cell(K_{X_z}) 
\\
&= \Cell(K_{X_z}).
\end{align*}
The isomorphism of groups (\ref{21d}(*) for $\eta = \xi$ and $x=y$) now follows from our clause \ref{hum1} of Tannaka duality. 

\ssegment{torsors on Friday}{}%%%
For the isomorphisms of path torsors, we claim that the triangle
\[
\xymatrix{
\DcoComod z^*B_x 
\ar@{<-}[dr]^\sim 
\ar[dd]_{\otimes (z^* {_\eta B_\xi}) } 
%{z^*()}
\\
&	D_\cell(K_{X_z})
\\
\DcoComod z^*B_y
\ar@{<-}[ur]_\sim
}
\]
commutes. Indeed, this amounts to the property of the bar construction
\[
B(M,A,x) \otimes_{B(x,A,x)}B(x,A,y) =
B(M,A,y)
\]
applied to $A = CX_z$ and $M$ a cellular complex of $K_{X_z}$-modules. We may then take $H^0$'s on the left and heart on the right, and apply part \ref{hum2} of Tannaka duality to obtain the hoped-for isomorphism. This completes the proof of theorem \ref{21d}.

Our application is a direct consequence of theorem \ref{21d}:

\begin{sthm}[Comparison of \'etale-internal and \'etale-external constructions]\label{ext-int}
Let $Z$ be a Dedekind scheme with function field $K$,
\[
f: X \to Z
\]
a smooth scheme over $Z$, $z = \Spec \overline K$ a generic geometric point of $Z$, $\xi, \eta$ sections of $f$ inducing sections $x, y$ of the pullback $f_z$ of $f$ to $z$ as in the the diagram of theorem \ref{21d}. We let $\et$ denote the pro-\'etale topology of Bhatt-Scholze \cite{ScholzeProet}. Let
\[
\m{Uni}(\Qp_{{X_z}_\et}) \xyto{\om_x^\un}{\om_y^\un} \Vect^c (\Qp)
\]
denote the fiber functors on the category of unipotent local systems on $X_z$ induced by $x$ and $y$. Then there's a $G_K$-equivariant isomorphism
\begin{align*}
\tag{*}
\Isom^\otimes & (\om_x^\un, \om_y^\un)
=
\\
&
\Spec z^*H^0 \big( 
\hocolim_{\dga(\Qp_{Z_\et})}
 ( 
  {\Qp}_{Z_{\et}}
  \xfrom{\eta^*}
 Rf_*{\Qp}_{X_\et} 
 \xto{\xi^*}
  {\Qp}_{Z_{\et}}
 ) 
 \big).
\end{align*}
\end{sthm}

\begin{proof}
The isomorphism furnished by theorem \ref{21d} is $G_K$-equivariant by functoriality. 
\end{proof}

We'll denote both sides of equation \ref{ext-int}(*) by ${_y \pi_x^\et}(X_{\overline K})$, or by $\pi_1^\et(X_{\overline K}, x)$ when $x = y$.\footnote{Let us emphasize that these are prounipotent completions of the classical profinite fundamental group and path torsors.} 

%%%%%%%%%%%%%%%%%%%%%
\section{h-Cotorsors and factorization of Kim's cutter}%%%%%
%%%%%%%%%%%%%%%%%%%

\segment{1611290918}{}%%%%%%%
We consider a smooth affine curve $X$ with \'etale divisor at infinity over an open integer scheme $Z \subset \Spec \Oo_K$, and a $Z$-point $x \in X(Z)$. We then have the loop ring $B_x$ of segment \ref{1611290926}. According to \S\ref{1611290928} and \ref{12.3} , $B_x$ has the structure of a commutative Hopf algebra object of the monoidal category $\DA = \DA(Z,\QQ)$. 

\segment{1611290919}{}%%%%%%
To define the notion of \emph{cotorsor} we wish to consider, we regard $B_x$ as a cogroup object of the category $\cMon \DA$ of commutative monoid objects in $\DA$. We define a \emph{pseudo-cotorsor under $B_x$} to be an object $P$ of $\cMon \DA$ equipped with a coaction map 
\[
B_x \otimes^\LL_\one P \from P
\]
such that 
\[
B_x \otimes^\LL_\one P \xfrom{\sim} P \otimes^\LL P;
\]
we say that $P$ is a \emph{cotorsor} if moreover  
\[
\tag{*}
H^0_\et(P_{\overline K}, \Qp) \neq 0.
\]
Here the symbol $\otimes^\LL_\one$ refers to the monoidal structure on $\cMon \DA$ induced by the monoidal structure $\otimes^\LL$ on $\DA$.

\segment{27auga}{}%%%
Assume $B_x$ is cellular. A pseudo-cotorsor is \emph{cellular} if the associated object of $\DA$ is cellular. In this case, the torsor condition \ref{1611290919}(*) is equivalent to
\[
\tag{*Cell}
H^0_t(P) \neq 0
\]
where $t$ refers to the cellular t-structure of \S\ref{celmot}. We let $\Cotors_{\cell} (B_x)$ denote the set of cellular cotorsors. We define 
\[
f_2: \Hom_{\DMdga}(\C(X), \one) \to
 \Cotors_{_\cell} (B_x)
\]
by
\[
f_2(\om) := {_\om B_x}.
\]

\segment{1611290920}{Lemma}%%%%%%
Suppose $B$ is a cogroup object of the category $\cMon \DA$ of commutative monoids in $\DA$. Fix $i \neq 0$.
\begin{enumerate}
\item
 Assume $H^i_\et(B_{\overline K}, \Qp) =0$.
 If $P$ is a cotorsor under $B$, then
 \[
 H^i_\et(P_{\overline K}, \Qp) = 0.
 \] 
 \item
 Assume $B$ is cellular and assume $H^i_t(B) = 0$. If $P$ is a cellular cotorsor under $B$, then
 \[
 H^i_t(P) = 0.
 \]
 \end{enumerate}

\begin{proof}
The two cases are similar; we restrict attention to (2). By functoriality of the Kunneth decomposition, we have an isomorphism
\[
0=H_t^i(B) \otimes H_t^0(P) \xfrom{\sim} H_t^i(P) \otimes H_t^0(P).
\]
Since $H_t^0(P) \neq 0 $ (\S \ref{27auga}), it follows that $H_t^i(P) = 0$ as claimed. 
\end{proof}

\segment{1611291405}{Mixed Tate case}%%%%%%
Assume $X$ is mixed Tate, so that the rational loop ring $B_x$ associated to our fixed $Z$-point $x$ is cellular, and the cellular t-structure of section \ref{tstrsec} is available. In view of the concentration theorem (\ref{ctr}), the pro-\textit{mixed Tate motive} $H^0_t(B_x)$ has the structure of a commutative Hopf algebra in $\Ind \TM^c(X,\QQ)$.  Following Deligne \cite{Deligne89}, we let $\Spec$ of a commutative monoid object denote itself regarded as an object of the opposite category. In this language, we set 
\[
\pi_1^\un(X,x) := \Spec H^0_t(B_x).
\]
Recall that the category of mixed Tate motives $\TM(Z,\QQ)$ possesses a canonical $\QQ$-rational fiber functor $\om_\can$ \cite[\S1.1]{DelGon}. We let 
\[
\pi_1^{\TM}(Z)
\]
denote the fundamental group of $\TM^c(Z,\QQ)$ at $\om_\can$. After applying $\om_\can$ to $H^0_t(B_x)$ we obtain a genuine Hopf algebra over $\QQ$ and we may take its spectrum in the usual sense:
\[
\pi_1^\un(X,x)^{\can} :=
 \Spec \om_\can 
 \big( 
 H^0_t(B_x)
 \big)
\]
to obtain a prounipotent $\QQ$-group equipped with an action of $\pi_1^{\TM}(Z)$. The group object $\pi_1^\un(X,x)$ and the group with $\pi_1^{\TM}(Z)$-action $\pi_1^\un(X,x)^\can$ merely package the same data in slightly different ways, and when there is no danger of confusion, we drop the superscript ``can''.\footnote{
We do not carry out the task of comparing our unipotent fundamental group with the one constructed by Deligne--Goncharov \cite{DelGon}, Esnault-Levine \cite{EsLevTate} and Levine \cite{LevineTMFG}. We note however that our group has all the needed properties (comparison with Tannakian \'etale and de Rham fundamental groups, relationship with motivic Ext groups) to carry out ``motivic Chabauty--Kim theory'', as in Dan-Cohen--Wewers \cite{mtmue}, Dan-Cohen \cite{CKIII}, Brown \cite{BrownUnit}, and Corwin--Dan-Cohen \cite{CorwinDCI, CorwinDCII}.
}

\ssegment{1611291401}{}%%%%%
In view of lemma \ref{1611290920}, the concentration theorem (\ref{ctr}), the cellular motives theorem (\ref{cellmot}), and interplay between Hopfoids and localization (\ref{12.3}), if $P$ is a cellular cotorsor under $B_x$, then $H^0_t(P)$ is a cotorsor under $H^0_t(B_x)$ in $\cMon \Ind \TM^c(Z, \QQ)$.  Thus, in terms of the canonical fiber functor $\om_\can$, we find that
\[
f_3(P):=\Spec \om_\can H^0_t(P)
\]
is a $\pi_1^{\TM}(Z)$-equivariant $\pi_1^\un(X,x)$-torsor; the equivariance means that there's an action of $\pi_1^{\TM}(Z)$ on $f_3(P)$ such that the action map
\[
f_3(P) \times \pi_1^\un(X,x) 
 \to
 f_3(P)
\]
is $\pi_1^{\TM}(Z)$-equivariant for the induced actions.

If $\om$ is any $\QQ$-rational fiber functor on $\TM(Z,\QQ)$, then the pointed set 
\[
H^1\big( \TM(Z, \QQ), \pi_1^\un(X,x) \big)
 :=
 H^1 \Big(
  \pi_1 \big( \TM(Z, \QQ), \om  \big), 
 	\om \big( \pi_1^\un(X,x) \big) 
	\Big)
\]
is independent of the choice of $\om$ (up to canonical bijection). Indeed, its elements are in bijection with cotorsors in $\cMon \Ind \TM(Z, \QQ)$.
We thus obtain a map of sets 
\[
f_3: \Cotors_{\cell} (B_x) \to
 H^1\big( \TM(Z, \QQ), \pi_1^\un(X,x) \big).
\]

\ssegment{}{}%%%%%%
For $\pP$ a prime of $Z$, we let $K_\pP$ denote the complete local field of $Z$ at $\pP$. Let $\om^\un_{x_\pP}$ denote the fiber functor on the category of unipotent $p$-adic local systems on $X_{\overline K_\pP}$ induced by $x$.  We define
\[
\mu_3^\pP:
X(Z_\pP) \to
H^1 \big(
G_{K_\pP}, 
\Aut^\otimes (\om_{x_{\pP}}^\un)
\big)
\]
by
\[
\mu_3^\pP(y) = \Isom^\otimes(\om_{x_\pP}^\un, \om_y^\un).
\]

\ssegment{}{}%%%%%%
If $P$ is a cotorsor for $B_x(X_\pP)$, then 
\[
H^0(P_{\overline K_\pP}, \Qp) :=
 H^0 (\m{Re}_{\overline K_\pP, \et} P )
\]
has the structure of a commutative algebra and a cotorsor in the category of $p$-adic $G_{K_\pP}$-representations. By the comparison theorem (\ref{ext-int}), its spectrum is thus a $G_{K_\pP}$-equivariant $\pi_1^\et(X_{\overline K_\pP}, x)$-torsor. We define
\[
f^{\pP}_3: \Cotors_\cell \big(B_x(X_\pP) \big)
 \to H^1 \big(
 G_{ K_{\pP}},
  \pi_1^\et(X_{\overline K_\pP}, x)
  \big) 
\]
by
\[
f^{\pP}_3(P) = \Spec H^0 
\big(
B_x(X_\pP)_{\overline K}, \Qp
\big).
\]

\ssegment{diabig}{}%%%%%%
We let $X_\pP$ denote the base-change of $X$ to $Z_\pP$. For $X$ mixed Tate, and $p$ a prime of $\ZZ$ lying below $Z$, we obtain a diagram like so,
\[
\xymatrix
{
%1
X(Z) \ar[r] \ar[d]_{f_1} &
		\prod_{\pP|p} X(Z_\pP) \ar[d]_{f_1^p}
		\ar@/^10pt/[d]^{\mu_1} 
		\ar@/^13ex/[dd]^{\mu_2}
		\ar@/^13ex/@<+7ex>[ddd]^{\mu_3}
		\ar@/^15ex/@<+9ex>[dddd]^{\mu^{[n]}}\\
%2
\Hom_{\DMdga}(\hrat(X), \one) 
		\ar[r]^-{l_1} \ar[d]_{f_2} &
		\prod_{\pP|p}
		\Hom_{\DMdga}(\hrat(X_\pP), \one)
		 \ar[d]_{f_2^p} 
		 \\
%3
\Cotors_\cell (B_x(X)) \ar[d]_{f_3} \ar[r]^-{l_2} &
		\prod_{\pP|p}
		\Cotors \big( B_x(X_\pP) \big) \ar[d]_{f_3^p} \\
%4
H^1(\pi_1^{\TM}(Z), \pi_1^\un(X))  \ar[r]^-{l_3} \ar[d]_{f_4} &
	\prod_{\pP | p} H^1
	\big( G_{K_\pP}
	, \pi_1^\et(X_{\overline K_\pP}, x) \big)
	\ar[d]_{f_4^p}
\\
%5
H^1(\pi_1^{\TM}(Z), \pi_1^\un(X)^{[n]})_{\Qp}  \ar[r]^-{l^{[n]}} &
	\prod_{\pP | p} H^1
	\big( G_{K_\pP}
	, \pi_1^\et(X_{\overline K_\pP}, x)^{[n]} \big),
}
\]
in which $\mu_3$ sends a point $y \in X(Z_\pP)$ to the Tannakian path torsor of unipotent \'etale paths from $x$ to $y$.

%crap below

\ssegment{diaproof}{}%%%%
We claim that the diagram of segment \ref{diabig} is commutative. The only nontrivial commutativities concern the small square bounded above by $l_2$, and the small triangle bounded on the left by $\mu_2$. We start with the former. Recall that there's a monoidal $p$-adic \'etale realization functor
\[
\m{Re}_\et: \DA(Z_p,\Qp) \to D({\Qp}_{{K_p}_\et})
\]
(\ref{J2}). The composite functor 
\[
\DA_\cell(Z, \Qp) \subset \DA(Z, \Qp) 
\to 
\DA(Z_p,\Qp) \to D({\Qp}_{{K_p}_\et})
\]
is compatible with t-structures. So there's a commuting pentagon 
\[
\xymatrix
@C=-1ex
{
\Cotors_\cell B_x \ar[d] \ar[r] &
\Cotors_\cell B_{x_{Z_p}} \ar[r] &
\Cotors B_{x_{K_p}}^\et \ar[d] \\
H^1\big(\TM(Z,\QQ), \pi_1^\un(X,x)\big) \ar[rr] &
&
H^1\big(G_{K_p}, \pi_1^\et(X,x) \big),
}
\]
whence the hoped-for square. 

We turn to the small triangle bounded on the left by $\mu_2$. In view of the commutativity established in section \ref{compsec}, it will suffice to construct a commuting square
\[
\xymatrix{
\Hom_{\DMdga(Z_\pP, \QQ)} (CX_\pP, \one)
\ar[r] \ar[d]
&
\Hom_{\Ddga(\Qp_{Z_{\pP, \et}})} (C_\et X_\pP, \Qp_{Z_{\pP, \et}})
\ar[d]
\\
\Cotors B_x(X_\pP)
\ar[r]
&
\Cotors B^\et_x(X_\pP).
}
\]
This commutativity is the same, mutatis mutandis, as the equality established in segment \ref{fal10} with the structured realization functor of segment \ref{ja1} in place of the unstructured realization functor.

\ssegment{}{}%%%%%
We set 
\[
 \left(\prod_{\pP|p} X(Z_\pP)\right)_
{\mbox{Kim, motivic}}
 := 
 \bigcap_n
 (\mu^{[n]})\inv(\overline {\Im} \; l^{[n]}).
\]

\segment{genc}{General case}%%%%%
If $X$ is not mixed Tate, we resort to a less motivic version of the above diagram. Let $S$ denote the complement of $Z$ in $\Spec \Oo_K$ and let $T$ be the union of $S$ and of the primes lying above $p$. Let $K_T$ be the maximal field extension of $K$ which is unramified outside of $T$ and let
\[
G_T := \Gal(K_T/K).
\]
We repeat the above construction with the nonabelian absolute Galois cohomology $H^1 \big( G_T, \pi_1^\et(X_{\overline K}, x) \big)$ (as well as its $[n]$-variant) in place of the nonabelian absolute motivic cohomologies appearing in the lower left corner. We denote the localization map
\[
H^1 \big( G_T, \pi_1^\et(X_{\overline K}, x) \big)
\to
	\prod_{\pP | p} H^1
	\big( G_{K_\pP}
	, \pi_1^\et(X_{\overline K_\pP}, x) \big)
\]
by $l_{3,\et}$, and its $[n]$-variant by $l_{\et}^{[n]}$. We then set  
\[
 \left(\prod_{\pP|p} X(Z_\pP)\right)_
{\mbox{Kim}}
 := 
 \bigcap_n
 (\mu^{[n]})\inv(\overline {\Im} \; l_\et^{[n]})
\]
where $\overline \Im$ denotes scheme-theoretic image.

\ssegment{fcond}{Remark}%%%%%
When $Z$ is an open subscheme of $\Spec \ZZ$ and $p \in Z$, our $X(\ZZ_p)_\m{Kim}$ is equal to the locus considered by Kim in \cite{kimii}, as we now explain. In loc. cit., Kim considers certain subspaces ``$H^1_\m f$'' of the Galois cohomology spaces which fit into a commuting diagram like so:
\[
\xymatrix{
%1
X(Z)
\ar[r] \ar[d]
&
X(\Zp)   \ar[d]^{\mu^{[n]}_\m f}
\\
%2
H^1_\m f \big( G_T, \pi_1^\et(X_{\Qbar}, x)^{[n]} \big)
\ar[r]^{l^{[n]}_{\et,\m f}}
 \ar@{}[d]|\bigcap
&
H^1_\m f \big(
 G_{\Qp},
 \pi_1^\et(X_{\overline \QQ_p}, x)^{[n]}
  \big)
 \ar@{}[d]|\bigcap
\\
%3
H^1 \big( G_T, \pi_1^\et(X_{\Qbar}, x)^{[n]} \big)
\ar[r]
&
H^1 \big(
 G_{\Qp},
 \pi_1^\et(X_{\overline \QQ_p}, x)^{[n]}
  \big)
}
\]
Moreover, it follows directly from the definitions that the bottom square is Cartesian. Consequently, Kim's locus
$
(\mu_\m f ^{[n]})\inv(\overline {\Im} \; l_{\et, \m f}^{[n]})
$
is equal to our
$
(\mu^{[n]})\inv(\overline {\Im} \; l_\et^{[n]}).
$
In other words, the crystalline condition signaled by the subscript ``f'' is not necessary for the construction of Kim's loci; its role, rather, is to render the maps $\mu^{[n]}$ computable.

\begin{sthm}[Factorization of Kim's cutter]%%%%
\label{factbig}
Suppose $Z$ is an open integer scheme and $X$ a smooth affine curve over $Z$ with \'etale divisor at infinity. Then in terms of 
the diagram of segment \ref{diabig}, with modifications as specified in segment \ref{genc}, we have a sequence of inclusions
\[
\xymatrix
{
X(Z)
\ar@{}[d]|\bigcap
\\
\left(\prod_{\pP|p} X(Z_\pP)\right)_
{\mbox{motivically global}}
 := \mu_{1}\inv(\Im l_{1})
 \ar@{}[d]|\bigcap
\\
\left(\prod_{\pP|p} X(Z_\pP)\right)_
{\mbox{pathwise motivically global}}
 := \mu_{2}\inv(\Im l_{2})
 \ar@{}[d]|\bigcap
\\
\left(\prod_{\pP|p} X(Z_\pP)\right)_{
\overset
{\mbox{pathwise motivically global}} 
{\mbox{up to $p$-adic \'etale homotopy}}
}
: = \mu_{3}\inv(\Im l_{3, \et})
 \ar@{}[d]|\bigcap
 \\
 \left(\prod_{\pP|p} X(Z_\pP)\right)_
{\mbox{Kim}}
.
 }
\]
Moreover, if $X$ is mixed Tate, we may replace ``$p$-adic \'etale homotopy'' by ``motivic homotopy'' (replace $l_{3,\et}$ by $l_{3}$) and we may replace the ``Kim locus'' by it's motivic version; together with the respective \'etale versions above, these form a lattice of inclusions like so (we drop the repeating symbols ``$\left(\prod_{\pP|p} X(Z_\pP)\right)$'' to save space):
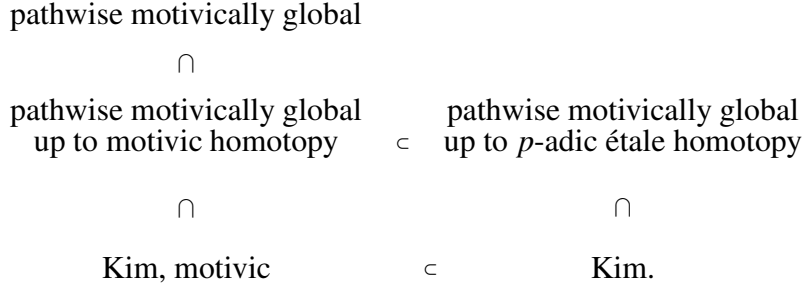
\[
\xymatrix
{
\mbox{pathwise motivically global}
 \ar@{}[d]|-\bigcap
\\
{
\overset
{\mbox{pathwise motivically global}} 
{\mbox{up to motivic homotopy}}
}
 \ar@{}[d]|-\bigcap
  \ar@{}[r]|-\subset
&
{
\overset
{\mbox{pathwise motivically global}} 
{\mbox{up to $p$-adic \'etale homotopy}}
}
 \ar@{}[d]|\bigcap
 \\
{\mbox{Kim, motivic}}
 \ar@{}[r]|-\subset
&
{\mbox{Kim}}
.
 }
\]
\end{sthm}

\begin{proof}
It remains only to point out that in the mixed Tate case, the \'etale  variant of the bottom portion of diagram \ref{diabig} fits together with the motivic version like so:
\small
\[
\xymatrix
@C=2ex
{
%4
H^1(\pi_1^{\TM}(Z), \pi_1^\un(X)) 
 \ar@/^3ex/[rr]^-{l_3} 
 \ar[d]
 \ar[r]
&
H^1 \big( G_T, \pi_1^\et(X_{\overline K}, x) \big)
\ar[r]_-{l_{3,\et}}
\ar[d]
&
	\prod_{\pP | p} H^1
	\big( G_{K_\pP}
	, \pi_1^\et(X_{\overline K_\pP}, x) \big)
	\ar[d]
\\
%5
H^1(\pi_1^{\TM}(Z), \pi_1^\un(X)^{[n]})_{\Qp}
  \ar@/_3ex/[rr]_-{l^{[n]}}
  \ar[r] 
&
H^1 \big( G_T, \pi_1^\et(X_{\overline K}, x)^{[n]} \big)
\ar[r]^-{l_\et^{[n]}}
&
	\prod_{\pP | p} H^1
	\big( G_{K_\pP}
	, \pi_1^\et(X_{\overline K_\pP}, x)^{[n]} \big).
}
\qedhere
\]
\end{proof}

\segment{}{}%%%%%%
We note that theorem \ref{thmfact} of the introduction (``factorization of Kim's cutter, preliminary version'') follows from its final version, theorem \ref{factbig}. Indeed, according to Kim \cite{kimii}, the finite part ``$H^1_f$'' of the local nonabelian \'etale homology of theorem \ref{factbig} (discussed in segment \ref{fcond}) and the de Rham fundamental group of theorem \ref{thmfact} are isomorphic. Moreover, this isomorphism interchanges the unipotent Kummer map (denoted $\mu_3$ in diagram \ref{diabig}) with the $p$-adic unipotent Albanese map described in segment \ref{unpalb}, as in the diagram
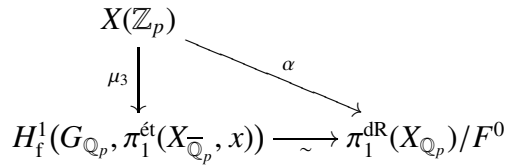
\[
\xymatrix{
X(\ZZ_p)
    \ar[dr]^{\al} \ar[d]_-{\mu_3}
\\
H^1_\m f \big(
 G_{\Qp},
 \pi_1^\et(X_{\overline \QQ_p}, x)
  \big)
  \ar[r]_-\sim
 &
 \pi_1^\dR(X_{\Qp})/F^0
}
\]
of Kim \cite{kimii}.

\segment{}{}%%%%%
Finally, we note that the ``pathwise motivically global theorem'' (\ref{pathglob}) follows from the concentration theorem (\ref{ctr}) in view of lemma \ref{1611290920}.

%%%%%%%%%%%%%%%%%%%%
%%%%%%%%%%%%%%%%%%%
\bibliography{RMP_Refs}%%%%%%

\begin{thebibliography}{10}

\bibitem{AyoubComodules}
Joseph Ayoub.
\newblock From motives to comodules over the motivic {H}opf algebra.
\newblock \url{http://user.math.uzh.ch/ayoub/PDF-Files/Motivic-Hopf-3.pdf}.

\bibitem{AyoubSixI}
Joseph Ayoub.
\newblock Les six op\'erations de {G}rothendieck et le formalisme des cycles
  \'evanescents dans le monde motivique. {I}.
\newblock {\em Ast\'erisque}, (314):x+466 pp. (2008), 2007.

\bibitem{AyoubSixII}
Joseph Ayoub.
\newblock Les six op\'erations de {G}rothendieck et le formalisme des cycles
  \'evanescents dans le monde motivique. {II}.
\newblock {\em Ast\'erisque}, (315):vi+364 pp. (2008), 2007.

\bibitem{AyoubEtale}
Joseph Ayoub.
\newblock La r\'ealisation \'etale et les op\'erations de {G}rothendieck.
\newblock {\em Ann. Sci. \'Ec. Norm. Sup\'er. (4)}, 47(1):1--145, 2014.

\bibitem{BalakDog}
J.~Balakrishnan and N.~Dogra.
\newblock Quadratic chabauty and rational points {I}: p-adic heights.
\newblock arXiv:1601.00388.

\bibitem{BalBesMulComp}
Jennifer Balakrishnan, Amnon Besser, and Steffen M{\"u}ller.
\newblock Computing integral points on hyperelliptic curves using quadratic
  chabauty.
\newblock {\em Math. of Computation.}, To appear.

\bibitem{BalBesMul}
Jennifer Balakrishnan, Amnon Besser, and Steffen M{\"u}ller.
\newblock Quadratic chabauty: p-adic height pairings and integral points on
  hyperelliptic curves.
\newblock {\em Journal f{\"u}r die reine und angewandte Mathematik}, To appear.

\bibitem{BalBesColGr}
Jennifer~S. Balakrishnan and Amnon Besser.
\newblock Coleman-gross height pairings and the p-adic sigma function.
\newblock {\em Journal f{\"u}r die reine und angewandte Mathematik (Crelle's
  journal)}.
\newblock To appear.

\bibitem{nabsd}
Jennifer~S. Balakrishnan, Ishai Dan-Cohen, Minhyong Kim, and Stefan Wewers.
\newblock A non-abelian conjecture of {T}ate--{S}hafarevich type for hyperbolic
  curves.
\newblock {\em Math. Ann.}, 372(1-2):369--428, 2018.

\bibitem{BalakAppendix}
Jennifer~S. Balakrishnan, Kiran~S. Kedlaya, and Minhyong Kim.
\newblock Appendix and erratum to ``{M}assey products for elliptic curves of
  rank 1'' [mr2629986].
\newblock {\em J. Amer. Math. Soc.}, 24(1):281--291, 2011.

\bibitem{BataninBerger}
M.~A. Batanin and C.~Berger.
\newblock Homotopy theory for algebras over polynomial monads.
\newblock {\em Theory Appl. Categ.}, 32:Paper No. 6, 148--253, 2017.

\bibitem{Besser}
Amnon Besser.
\newblock Coleman integration using the {T}annakian formalism.
\newblock {\em Math. Ann.}, 322(1):19--48, 2002.

\bibitem{ScholzeProet}
Bhargav Bhatt and Peter Scholze.
\newblock The pro-{\'e}tale topology for schemes.
\newblock arXiv:1309.1198, 2013.

\bibitem{BousfieldGugenheim}
A.~K. Bousfield and V.~K. A.~M. Gugenheim.
\newblock On {${\rm PL}$} de {R}ham theory and rational homotopy type.
\newblock {\em Mem. Amer. Math. Soc.}, 8(179):ix+94, 1976.

\bibitem{BrownUnit}
Francis C.~S. Brown.
\newblock Integral points on curves, the unit equation, and motivic periods.
\newblock {\em arXiv:1704.00555}, 2017.

\bibitem{ChatUnv}
Andre Chatzistamatiou and Sinan {\"U}nver.
\newblock On {$p$}-adic periods for mixed {T}ate motives over a number field.
\newblock {\em Math. Res. Lett.}, 20(5):825--844, 2013.

\bibitem{ChiarLeStum}
Bruno Chiarellotto and Bernard Le~Stum.
\newblock {$F$}-isocristaux unipotents.
\newblock {\em Compositio Math.}, 116(1):81--110, 1999.

\bibitem{CDHomAlg}
Cisinski and D\'eglise.
\newblock Local and stable homological algebra in {G}rothendieck abelian
  categories.
\newblock {\em Homology, Homotopy and Applications}, 11(1):219--260, 2009.

\bibitem{Coleman}
Robert~F. Coleman.
\newblock Dilogarithms, regulators and {$p$}-adic {$L$}-functions.
\newblock {\em Invent. Math.}, 69(2):171--208, 1982.

\bibitem{CorwinDCI}
David Corwin and Ishai Dan-Cohen.
\newblock The polylog quotient and the {G}oncharov quotient in computational
  {C}abauty-{K}im theory {I}.
\newblock {\em International Journal of Number Theory, to appear}.

\bibitem{CorwinDCII}
David Corwin and Ishai Dan-Cohen.
\newblock The polylog quotient and the {G}oncharov quotient in computational
  {C}abauty-{K}im theory {II}.
\newblock {\em Transactions of the AMS, to appear}.

\bibitem{CKIII}
Ishai Dan-Cohen.
\newblock Mixed {T}ate motives and the unit equation {II}.
\newblock {\em Algebra and Number Theory, to appear}.

\bibitem{nmh}
Ishai Dan-Cohen and Tomer Schlank.
\newblock Morphisms of rational motivic homotopy types.
\newblock arXiv:1811.06365.

\bibitem{CKtwo}
Ishai Dan-Cohen and Stefan Wewers.
\newblock Explicit {C}habauty-{K}im theory for the thrice punctured line in
  depth 2.
\newblock {\em Proc. Lond. Math. Soc. (3)}, 110(1):133--171, 2015.

\bibitem{mtmue}
Ishai Dan-Cohen and Stefan Wewers.
\newblock Mixed {T}ate motives and the unit equation.
\newblock {\em Int. Math. Res. Not. IMRN}, (17):5291--5354, 2016.

\bibitem{DelGon}
P.~Deligne and A.~B. Goncharov.
\newblock Groupes fondamentaux motiviques de {T}ate mixte.
\newblock {\em Ann. Sci. \'Ecole Norm. Sup. (4)}, 38(1):1--56, 2005.

\bibitem{Deligne89}
Pierre Deligne.
\newblock Le groupe fondamental de la droite projective moins trois points.
\newblock In {\em Galois groups over {${\bf Q}$} ({B}erkeley, {CA}, 1987)},
  volume~16 of {\em Math. Sci. Res. Inst. Publ.}, pages 79--297. Springer, New
  York, 1989.

\bibitem{EsLevTate}
H{\'e}l{\`e}ne Esnault and Marc Levine.
\newblock Tate motives and the fundamental group.
\newblock {\em arXiv:0708.4034}.

\bibitem{HinichRectification}
Vladimir Hinich.
\newblock Rectification of algebras and modules.
\newblock {\em Doc. Math.}, 20:879--926, 2015.

\bibitem{HoveySpectra}
Mark Hovey.
\newblock Spectra and symmetric spectra in general model categories.
\newblock {\em J. Pure Appl. Algebra}, 165(1):63--127, 2001.

\bibitem{Iwanari}
Isamu Iwanari.
\newblock Bar construction and tannakization.
\newblock arXiv:1203.0492, 2012.

\bibitem{kimi}
Minhyong Kim.
\newblock The motivic fundamental group of {$\mathbb P^1\setminus
  \{0,1,\infty\}$} and the theorem of {S}iegel.
\newblock {\em Invent. Math.}, 161(3):629--656, 2005.

\bibitem{kimii}
Minhyong Kim.
\newblock The unipotent {A}lbanese map and {S}elmer varieties for curves.
\newblock {\em Publ. Res. Inst. Math. Sci.}, 45(1):89--133, 2009.

\bibitem{KimMassey}
Minhyong Kim.
\newblock Massey products for elliptic curves of rank 1.
\newblock {\em J. Amer. Math. Soc.}, 23(3):725--747, 2010.

\bibitem{KimTangential}
Minhyong Kim.
\newblock Tangential localization for {S}elmer varieties.
\newblock {\em Duke Math. J.}, 161(2):173--199, 2012.

\bibitem{KimHain}
Minhyong Kim and Richard~M. Hain.
\newblock A de {R}ham-{W}itt approach to crystalline rational homotopy theory.
\newblock {\em Compos. Math.}, 140(5):1245--1276, 2004.

\bibitem{LevineVanishing}
Marc Levine.
\newblock Tate motives and the vanishing conjectures for algebraic
  {$K$}-theory.
\newblock In {\em Algebraic {$K$}-theory and algebraic topology ({L}ake
  {L}ouise, {AB}, 1991)}, volume 407 of {\em NATO Adv. Sci. Inst. Ser. C Math.
  Phys. Sci.}, pages 167--188. Kluwer Acad. Publ., Dordrecht, 1993.

\bibitem{LevineMM}
Marc Levine.
\newblock {\em Mixed motives}, volume~57 of {\em Mathematical Surveys and
  Monographs}.
\newblock American Mathematical Society, Providence, RI, 1998.

\bibitem{LevineTMFG}
Marc Levine.
\newblock Tate motives and the fundamental group.
\newblock In {\em Cycles, motives and {S}himura varieties}, Tata Inst. Fund.
  Res. Stud. Math., pages 265--392. Tata Inst. Fund. Res., Mumbai, 2010.

\bibitem{LurieAlg}
Jacob Lurie.
\newblock Higher algebra.
\newblock http://www.math.harvard.edu/~lurie/.

\bibitem{LurieTopos}
Jacob Lurie.
\newblock {\em Higher topos theory}, volume 170 of {\em Annals of Mathematics
  Studies}.
\newblock Princeton University Press, Princeton, NJ, 2009.

\bibitem{OlssonBar}
Martin Olsson.
\newblock The bar construction and affine stacks.
\newblock {\em Comm. Algebra}, 44(7):3088--3121, 2016.

\bibitem{OlssonTowards}
Martin~C. Olsson.
\newblock Towards non-abelian {$p$}-adic {H}odge theory in the good reduction
  case.
\newblock {\em Mem. Amer. Math. Soc.}, 210(990):vi+157, 2011.

\bibitem{Positselski}
Leonid Positselski.
\newblock Two kinds of derived categories, {K}oszul duality, and
  comodule-contramodule correspondence.
\newblock {\em Mem. Amer. Math. Soc.}, 212(996):vi+133, 2011.

\bibitem{PridhamGalois}
J.~P. Pridham.
\newblock Galois actions on homotopy groups of algebraic varieties.
\newblock {\em Geom. Topol.}, 15(1):501--607, 2011.

\bibitem{pridhamenhanced}
J.~P. Pridham.
\newblock Tannaka duality for enhanced triangulated categories.
\newblock arXiv:1309.0637, 2013.

\bibitem{RobaloBridge}
Marco Robalo.
\newblock {$K$}-theory and the bridge from motives to noncommutative motives.
\newblock {\em Adv. Math.}, 269:399--550, 2015.

\bibitem{SchlankGalDu}
Tomer Schlank and Vesna Stojanoska.
\newblock Local galois duality for spectra.
\newblock In preparation.

\bibitem{SchlankPT}
Tomer Schlank and Vesna Stojanoska.
\newblock Poitou tate duality for spectra.
\newblock In preparation.

\bibitem{SchwedeShipley}
Stefan Schwede and Brooke~E. Shipley.
\newblock Algebras and modules in monoidal model categories.
\newblock {\em Proc. London Math. Soc.}, 80(3):491--511, 2000.

\bibitem{ToenAffine}
Bertrand To\"en.
\newblock Champs affines.
\newblock {\em Selecta Math. (N.S.)}, 12(1):39--135, 2006.

\bibitem{VoevTrCat}
Vladimir Voevodsky.
\newblock Triangulated categories of motives over a field.
\newblock In {\em Cycles, transfers, and motivic homology theories}, volume 143
  of {\em Ann. of Math. Stud.}, pages 188--238. Princeton Univ. Press,
  Princeton, NJ, 2000.

\bibitem{Weibel}
Charles~A. Weibel.
\newblock {\em An introduction to homological algebra}, volume~38 of {\em
  Cambridge Studies in Advanced Mathematics}.
\newblock Cambridge University Press, Cambridge, 1994.

\bibitem{White}
David White.
\newblock Model structures on commutative monoids in general model categories.
\newblock {\em Journal of pure and applied algebra}.
\newblock To appear.

\bibitem{Wojtkowiak}
Zdzis{\l}aw Wojtkowiak.
\newblock Cosimplicial objects in algebraic geometry.
\newblock In {\em Algebraic {$K$}-theory and algebraic topology ({L}ake
  {L}ouise, {AB}, 1991)}, volume 407 of {\em NATO Adv. Sci. Inst. Ser. C Math.
  Phys. Sci.}, pages 287--327. Kluwer Acad. Publ., Dordrecht, 1993.

\end{thebibliography}
%%%%%%%%%%%%%%%%%%
\bibliographystyle{plain}%%%
%%%%%%%%%%%%%%%%%
%%%%%%%%%%%%%%%%%

\vfill

\noindent\Small\textsc{
ID: Department of Mathematics \\ Ben-Gurion University of the Negev.
\\
} 
\texttt{ishaidc@gmail.com}

\bigskip

\noindent\Small\textsc{
TS: Einstein Institute of Mathematics \\ Hebrew University of Jerusalem. 
}
\\
\texttt{tomer.schlank@mail.huji.ac.il}

\end{document}